\documentclass[a4paper,11pt]{amsart}
\pdfoutput=1
\usepackage[pdfauthor={Anders S\''odergren, 
Morten S. Risager}, pdftitle={}, pdfkeywords={Angles, hyperbolic
lattice point}, 
pdfcreator={Pdflatex},colorlinks=true]{hyperref}

\usepackage[french,english]{babel}
\renewenvironment{abstract}[1]
  {\selectlanguage{#1}%
   \quotation \footnotesize\textsc\abstractname.}
  {\endquotation \bigskip}

\usepackage{amsmath,amssymb}
\usepackage{todonotes}
\usepackage{appendix}
\usepackage{enumerate}
\newtheorem{theorem}{Theorem}[section]
\newtheorem{lemma}[theorem]{Lemma}
\newtheorem{proposition}[theorem]{Proposition} 
\theoremstyle{definition}

\theoremstyle{remark}
\newtheorem{remark}[theorem]{Remark}
\newtheorem{corollary}[theorem]{Corollary}   

\newcommand{\norm}[1]{\left\lVert #1 \right\rVert}
\newcommand{\vol}[1]{\textup{vol}( #1 )}
\newcommand{\abs}[1]{\left\lvert #1 \right\rvert}
\renewcommand{\v}[2]{{v}( #1, #2)}
\newcommand{\inprod}[2]{\left \langle #1,#2 \right\rangle}

\def \G {{\Gamma}}
\def \R {{\mathbb R}}
\def \H {{\mathbb H}}
\def \C {{\mathbb C}}
\def \Z {{\mathbb Z}}
\def \g {{\gamma}}
\def \N {{\mathbb N}}

\def \dim {{\hbox{dim}}}

\numberwithin{equation}{section}

\author{Morten S. Risager}
\address{Department of Mathematical
  Sciences, University of Copenhagen, Universitetsparken 5, 2100
  Copenhagen \O, Denmark}
\email{risager@math.ku.dk}
\author{Anders S\"odergren}
\address{Department of Mathematical
  Sciences, University of Copenhagen, Universitetsparken 5, 2100
  Copenhagen \O, Denmark}
\email{sodergren@math.ku.dk}
\title[Angles in hyperbolic lattices]{Angles in hyperbolic lattices:
  The pair correlation density}
\date{\today}
\thanks{The first author was supported by a Sapere Aude grant from The
  Danish Council for Independent Research. The second author was
  supported by a grant from The Danish Council for Independent
  Research and FP7 Marie Curie Actions-COFUND (grant id:
  DFF-1325-00058).}
\begin{document} 
\maketitle
\vspace{-.41cm}
\begin{abstract}{english}
It is well known  that the angles in a lattice acting on hyperbolic $n$-space become equidistributed. In this paper we determine a formula for the pair correlation density for angles in such hyperbolic lattices. Using this formula we determine, among other things, the asymptotic behavior of the  density function in both the small and large variable limits. This extends earlier results by Boca, Pa{\c{s}}ol, Popa and Zaharescu and Kelmer and Kontorovich in dimension 2 to general dimension $n$. Our proofs use the decay of matrix coefficients together with a number of careful estimates, and lead to effective results with explicit rates. 
\end{abstract}
\section{Introduction}
Let $\G$ be a discrete co-finite subgroup of  $G=\hbox{SO}_0(n,1)$ and
let $z_0\in \H^n$. In its most basic form the hyperbolic lattice point
counting problem seeks
to estimate the size of the orbit $\Gamma z_0$ inside some expanding
region. This
problem has been studied -- when the region is a hyperbolic ball  --
by several people
\cite{Delsarte:1942a,Huber:1956,Huber:1959a,Patterson:1975a,Gunther:1980a,
  LaxPhillips:1982a, Good:1983b,
  Levitan:1987a, PhillipsRudnick:1994a}, and precise asymptotics
are known. By now it is also well-established that the angles in a hyperbolic lattice
are equidistributed \cite{Nicholls:1983a, Good:1983b,Boca:2007a, RisagerTruelsen:2010, GorodnikNevo:2012}. 

More refined angle statistics has been studied only recently: Boca,
Pa{\c{s}}ol, Popa and Zaharescu \cite{BocaPasolPopaZaharescu:2014, BocaPopaZaharescu:2013} studied the pair correlation
statistics for hyperbolic angles when $\Gamma=\hbox{PSL}_2(\Z)$ and
$z_0=i$ or $z_0=e^{i\pi/3}$, and gave conjectures for all lattices
$\G\subseteq \hbox{PSL}_2(\R)\cong \hbox{SO}_0(2,1)$ and all base
points $z_0$. These conjectures were later resolved by Kelmer and Kontorovich \cite{KelmerKontorovich:2013}.

In this  paper, among other things, we generalize and extend these
results to general
lattices acting on $n$-dimensional hyperbolic space. To be precise:
Let $z_0=e_{n+1}\in \H^n$ be the origin of (the hyperboloid model of)
$\H^n$. Consider
$$\norm{g}^2=2\cosh(d(ge_{n+1},e_{n+1})),$$ 
where $d(z,w)$ denotes
the hyperbolic distance between $z,w\in \H^n$, and let $v(g,g')$ denote
the hyperbolic angle between $ge_{n+1}$ and $g'e_{n+1}$ (see Section
\ref{sec:basicangles}). Let $K=\hbox{Stab}_G(e_{n+1})$. For simplicity we
assume that $\Gamma\cap K$ is trivial,  although this is not a serious restriction. We define
\begin{equation*}
N_\G(Q):=\G\cap B_Q,
\end{equation*}
where $B_Q:=\{g\in G : \norm{g}\leq Q\}$. For the hyperbolic lattice
point problem the main interest is in the asymptotic behavior of $\#N_\G(Q)$ as $Q\to \infty$. In
the above terminology Lax and Phillips proved \cite[Thm.\
1]{LaxPhillips:1982a} (see also \cite[Thm.\ II]{Levitan:1987a},
\cite[Thm.\ 4.1]{ElstrodtGrunewaldMennicke:1988a}, and
Remark \ref{LPremark}) that
\begin{equation}\label{LP-intro}\#N_\Gamma(Q)=\frac{\vol{B_{Q}}}{\vol{\Gamma\backslash
    G}}+O\big(\vol{B_{Q}}^{1-\delta}\big), \textrm{ as }Q\to\infty.
\end{equation} 
Here $\vol{B_Q}$ is the (appropriately normalized) Haar measure of $B_Q\subseteq G$ (see Section \ref{Haar-section}).  The size of $\delta$ usually depends either directly or indirectly on a spectral gap, i.e. the size of the least non-zero element in the spectrum of the automorphic Laplacian.

 Consider an element $\g'\in \G$ and a real number $\xi>0$. By a short
heuristic argument based on known equidistribution and point counting results (see Section \ref{heuristics}), we expect about
$\xi^{n-1}$ elements in the set 
\begin{equation*}
\left\{\g\in N_\G(Q)\backslash\{\g'\} :  \v{\g}{\g'}<
\frac{2k_{n,\G}}{Q^{2}}\xi\right\}.
\end{equation*}
Here $k_{n,\G}$ is an explicit constant (see \eqref{kn-constant}) which we compute as part of
the heuristics. 
Taking averages over $\g'$ with $\norm{\g'}\leq Q$, we are led to
investigate the pair correlation counting function
\begin{equation}\label{pair-correlation-function-intro}
R_{2,Q}(\xi):=\frac{\#N_{2,Q}(\xi)}{\# N_\G(Q)},
\end{equation}
where 
\begin{equation}
\label{enum}  N_{2,Q}(\xi):=\left\{\g,\g'\in N_\G(Q):\gamma^{-1}\gamma'\notin K, \v{\g}{\g'}< \frac{2k_{n,\G}}{Q^{2}}\xi\right\}.
\end{equation}
We note that without the condition $\gamma^{-1}\gamma'\notin K$ in the
definition of $N_{2,Q}(\xi)$, the number $R_{2,Q}(\xi)$ increases by
exactly 1 since $\G\cap K$ is trivial.

We want to investigate how much $R_{2,Q}(\xi)$ deviates from $\xi^{n-1}$
in the limit as $Q\to \infty$. Our first result is the following theorem which asserts that the limit
as $Q\to \infty$
does indeed exist:

\begin{theorem}\label{main theorem}
Let $n\geq2$ and let $\G\subseteq G$ be a lattice as above.  Then,
as $Q\to \infty$, the function $R_{2,Q}(\xi)$ converges to a
differentiable limit  $R_2(\xi)$ whose derivative satisfies
\begin{align*}\label{g2formula preliminary}
g_2(\xi)=\frac{d}{d\xi}R_2(\xi)=\sum_{M\in\Gamma}F_{\xi}(d(Me_{n+1},e_{n+1})),
\end{align*}
where $F_\xi$ is given explicitly in \eqref{Fxi}. 
Moreover,  there exists $\nu>0$ such that  
\begin{equation*}
R_{2,Q}(\xi)=R_{2}(\xi)+O_\xi(Q^{-\nu}).
\end{equation*}
\end{theorem} 

\begin{remark}
The limit function $R_2(\xi)$ is called \emph{the pair correlation
function}, and its derivative $g_2(\xi)$ is called \emph{the pair
  correlation density}. An explicit estimate on the rate of
convergence (i.e.  $\nu$ in Theorem \ref{main theorem}) is given in terms of a spectral gap:  We write the eigenvalues of the Laplace-Beltrami operator $-\Delta$ on $\G\backslash G$ in increasing order as $0=\lambda_0<\lambda_1\leq\lambda_2\leq\ldots$ and choose 
$s_0\in(\frac{n-1}{2},n-1)$ so that $\lambda_1>s_0(n-1-s_0)$. The size
of $\nu$ is directly related to the size of $n-1-s_0$.  We refer to
Theorem \ref{main theorem precise version} for details. For $n=2$ our
convergence rate is identical to that proved in \cite{KelmerKontorovich:2013}.
\end{remark}

The fact that the function $F_\xi$ in Theorem \ref{main theorem} can be given explicitly  (see
  \eqref{Fxi}, \eqref{DEFOFF}, and Remark \ref{n=2,3}), allows us to determine
 the asymptotic behavior of the pair correlation density:

\begin{theorem}\label{main theorem-asymptotics}
 Fix $s_0$ as above. We have
\begin{equation*}
g_2(\xi)=(n-1)\xi^{n-2}+O\left(\xi^{n-2+\frac{2(s_0-n+1)}{n+1}}\right),
\quad \textrm{as $\xi\to\infty$.}
\end{equation*}
\end{theorem}

\begin{remark}
Integrating the asymptotic formula in Theorem \ref{main
  theorem-asymptotics}, we immediately find that
\begin{equation*}
R_2(\xi)=\xi^{n-1}+O\left(\xi^{n-1+\frac{2(s_0-n+1)}{n+1}}\right),
\quad \textrm{ as $\xi\to \infty$}.
\end{equation*}
\end{remark}

Turning now instead to the limit $\xi\to 0$, Kelmer and Kontorovich
\cite{KelmerKontorovich:2013} proved that in the 2-dimensional case
the pair correlation density tends to a  \emph{non-zero} value. Using the above explicit description of $g_2$, we show that this happens \emph{only} in this case.

\begin{theorem}\label{g_2at0theorem}
The pair correlation density converges to zero as $\xi$ tends to zero
if and only if $n\neq 2$.
\end{theorem}

We observe that the pair correlation function $R_2(\xi)$  depends
heavily both on the discrete group $\G$ and the choice of base point
for the lattice point problem.\footnote{Recall that a change of base point in our problem can be achieved by conjugating the group $\G$.} However, once the group and the base point are fixed, the pair correlation function is \emph{uniform} in the following sense: Let  $\mathcal U$ denote the hyperbolic unit sphere centered at $e_{n+1}$. Let $\mathcal S\subset \mathcal U$ be a spherical cap and define $\mathcal C$ to be the hyperbolic cone specified by the vertex $e_{n+1}$ and the cross-section $\mathcal S$. Then, if we restrict our attention in (both the numerator and the denominator of) \eqref{pair-correlation-function-intro} to elements in $\G$ corresponding to points in the orbit $\G e_{n+1}$ lying in $\mathcal C$, then the limit as $Q\to\infty$ still exists and equals the same function $R_2(\xi)$ achieved in Theorem \ref{main theorem}. In order to give a precise statement, we define
\begin{equation*}
N_{\G,\mathcal C}(Q):=\{\gamma\in N_{\G}(Q) : \gamma e_{n+1}\in \mathcal C\}
\end{equation*}
and  
\begin{align}\label{pair-correlation-function-cone}
N_{2,\mathcal C,Q}(\xi):=\left\{\g,\g'\in N_{\G,\mathcal C}(Q) : \gamma^{-1}\gamma'\notin K,  \v{\g}{\g'}< \frac{2k_{n,\G}}{Q^{2}}\xi\right\}.\nonumber
\end{align}      
    
\begin{theorem}\label{restricted-main-theorem}
Let $n\geq 2$ and let $\G\subset G$ be a lattice as above. Let $\mathcal S\subset\mathcal U$ be a spherical cap and denote the hyperbolic cone specified by the vertex $e_{n+1}$ and the cross-section $\mathcal S$ by $\mathcal C$. Then  
\begin{equation*}
\lim_{Q\to\infty}\frac{\#N_{2,\mathcal C,Q}(\xi)}{\#N_{\G,\mathcal
    C}(Q)}
=R_2(\xi).
\end{equation*}
In particular the limit exists, is differentiable with derivative $g_2(\xi)$, and is independent of the cone $\mathcal C$. 
\end{theorem}

\begin{remark}
It is clear that our techniques can handle also pair correlation functions corresponding to cones $\mathcal C$ specified by more general sets $\mathcal S\subset\mathcal U$. However, for simplicity, we have chosen not to give the most general statement possible.
\end{remark}

\begin{remark} The function $F_\xi(l)$ -- which by Theorem \ref{main
    theorem} and Theorem \ref{restricted-main-theorem} determines the
  pair correlation density $g_2$ both in the whole space and in
  sectors -- depends, apart from $\xi$ and $l$,  only on $n$ and $\vol{\Gamma\backslash G}$. It follows that $g_2$ depends on the group $\G$ only through the sequence $$\{d(\gamma e_{n+1},e_{n+1}):\gamma \in \G\}.$$ In fact, it turns out that this sequence determines and is determined  by $g_2$ and the volume $\vol{\G\backslash G}$. Using this we show, in Section \ref{friday-afternoon}, that $g_2$ determines and is determined by certain spectral data. We refer to Section \ref{friday-afternoon} for precise statements.
\end{remark}

The idea of the proof of Theorem
  \ref{main theorem} (which is the basis for most of the subsequent
  results) is as
  follows: Considering the results in
\cite{BocaPasolPopaZaharescu:2014, BocaPopaZaharescu:2013,
  KelmerKontorovich:2013}, we expect the pair correlation density to
be expressible as a sum over
$M=\g^{-1}\g'$. We therefore write  $\#N_{2,Q}(\xi)$  (see
\eqref{enum}) as  
\begin{equation}\label{basic-counting}
\#N_{2,Q}(\xi)=\sum_{\substack{M\in \G\\ M\notin K}}\#\G\cap \mathcal R_M(Q,k_{n,\G}\xi),
\end{equation}
where 
\begin{equation*}
\mathcal R_M(Q,\xi):=\left\{g\in B_Q: \norm{gM}\leq Q,
\v{g}{gM}< \frac{2\xi}{Q^2}\right\}.
\end{equation*}
Our goal (following \cite{KelmerKontorovich:2013}) is then to show that
 the number
$\#\G\cap \mathcal R_M(Q,k_{n,\G}\xi)$ can be approximated by
$\vol{R_M(Q,k_{n,\G}\xi)}/\vol{\G\backslash G}$, and to compute  approximations
for $\vol{R_M(Q,k_{n,\G}\xi)}$. 

  The structure of the paper is as follows: In Section \ref{prereq} we
  review known theory and results needed in the proofs of the main
  theorems. In Section \ref{stability} we show several stability results
  for angles and norms, which are later used for certain approximation
  arguments. In Section \ref{vols} we find expressions for
  $\vol{R_M(Q,k_{n,\Gamma}\xi)}$ given in terms of the functions
  $F_\xi(d(Me_{n+1},e_{n+1}))$, and in Section \ref{counts-to-vols} we
    show how these volumes are related to $\#\G\cap \mathcal
    R_M(Q,k_{n,\G}\xi)$. In Section \ref{proofs}  we tie these investigations
    together and complete the proofs of all the main theorems.  Several of
    these results use basic properties of the function $F_\xi$, and
    we state and prove such properties in Appendix \ref{appendix}.

In a very recent paper, Marklof and Vinogradov
\cite{MarklofVinogradov:2014} show how the
mixing property of the geodesic flow can be used to obtain information
about the distribution of directions in a hyperbolic lattice. They show convergence of all mixed
moments of the appropriate counting functions, which in particular allows them to conclude
convergence of pair correlations. In fact their results capture all local statistics (e.g. gap or nearest neighbor
distributions).  In the present paper, we
 focus on the pair correlation and use information on the decay of matrix coefficients to get more explicit and precise results.    

\subsection*{Notation}

Throughout this manuscript we consider $n$ and $\Gamma$ as fixed. In
all estimates, the implied constants may depend on $n$ and $\Gamma$;
any other dependence will be specified.

\section{Prerequisites}\label{prereq}

\subsection{The hyperboloid model of hyperbolic $n$-space}

Let $n\geq2$. We begin by recalling that
$$\hbox{SO}(n,1)=\left\{g\in\hbox{SL}_{n+1}(\R) : g^tJg=J\right\},$$ 
where $J=\hbox{diag}(I_n,-1)$ and $I_n$ is the $n\times n$ identity
matrix. By definition this is the subgroup of $\hbox{SL}_{n+1}(\R)$ leaving the (symmetric and
non-degenerate) bilinear form 
$$\inprod{x}{y}=x^tJy\qquad (x,y\in\R^{n+1})$$ 
invariant. In the present paper we will be mainly interested in the group $G:=\hbox{SO}_0(n,1)$ defined as the connected component of $\hbox{SO}(n,1)$ containing the identity. 

The group $G$ acts transitively by matrix multiplication on the set
$$\H^n:=\left\{x\in \R^{n+1} : \inprod{x}{x}=-1,\, x_{n+1}>0\right\},$$ 
which is the upper sheet of a two-sheeted hyperboloid.  If we define the metric $d$ on $\H^n$ by the relation 
\begin{equation*}
\cosh d(x,y)=-\inprod{x}{y},
\end{equation*}
then $\H^n$ is a model of hyperbolic $n$-space, i.e.\ a
maximally symmetric,  simply connected Riemannian manifold of
dimension $n$ and constant sectional curvature $-1$. In this model, the group
$G$ acts as the full group of orienting preserving isometries on hyperbolic $n$-space.

\subsection{Cartan decomposition} 

Consider the groups
\begin{equation*}
K:=\left\{\left( \begin{array}{cc}
k' &  \\
 & 1  \end{array} \right) : k'\in \hbox{SO}(n)\right\}
\end{equation*}
and
\begin{equation*}
A:=\left\{a_t =\left( \begin{array}{ccc}
\cosh t & &\sinh t \\
&I_{n-1}&\\
\sinh t & &\cosh t  \end{array} \right) : t\in \R\right\}.
\end{equation*}
 The Cartan decomposition of $G$ gives that $G=KA^+K$, where $A^+:=\{a_t\in A : t\geq0\}$. In particular every element $g=(g_{i,j})\in G$ can be written as 
\begin{equation}
g=k_g a_{t(g)}k_g'\label{cartan}
\end{equation}
where $k_g,k_g'\in K$ and  $a_{t(g)}\in A^+$. Here\footnote{Here and throughout, we let $e_{j}$ denote the $j$th standard basis vector in $\R^{n+1}$.} 
$$t(g)=\cosh^{-1}(g_{n+1,n+1})=d(ge_{n+1},e_{n+1}).$$
We note that $K$ is the stabilizer of $e_{n+1}\in \H^n$. 

The decomposition \eqref{cartan} is not unique as the centralizer of $A$ in
$K$, $M:=Z_K(A)$, is non-trivial. In concrete terms:
\begin{equation*}
M=\left\{\left( \begin{array}{ccc}
1 & & \\
&m'&\\
& &1  \end{array} \right) : m'\in \hbox{SO}(n-1)\right\}\subset K.
\end{equation*}
An element $g\in G$ can be written as $g=k_g
a_{t(g)}k'_g$ and $g=\tilde k_g a_{\tilde t(g)}\tilde k'_g $, if and only if
$t(g)=\tilde t(g)$, $\tilde k_g=k_gm$ and $\tilde k_g'=m^{-1}k_g'$ for
some $m\in M$. 

\subsection{Basic properties of hyperbolic angles}\label{sec:basicangles}

The formula for the hyperbolic angle between two geodesic segments intersecting at a point $z\in\H^n$ is in general quite complicated (see, e.g., \cite[Sect.\ 3.2]{Ratcliffe:1994a}). However, we will restrict our attention to the case where the vertex of the angle is located at the point $e_{n+1}\in\H^n$, which simplifies the formulas considerably. 

For $x,y\in \H^n\backslash\{e_{n+1}\}$, we define the corresponding (unsigned) angle $\v{x}{y}\in[0,\pi]$, based at the point $e_{n+1}$, via the relation
\begin{equation*}
\cos(\v{x}{y})=\frac{u\cdot v}{\sqrt{u\cdot u}\sqrt{v\cdot v}},
\end{equation*}
where $x=(u,t), y=(v,s)$ for appropriate choices of $u,v\in \R^n$, $s,t\in \R$, and $\cdot$ is
the usual Euclidean inner product on $\R^n$. In addition, for $g,g'\in G\setminus K$, we define, by an abuse of notation, the
angle between them  by
\begin{equation*}
\v{g}{g'}:=\v{g e_{n+1}}{g' e_{n+1} }.
\end{equation*}

Fixing the point $N:= (1,0,\cdots ,0,\sqrt{2})^t\in\H^n$, 
it is straightforward to verify that if $g=ka_tk'$ with $k,k'\in K$, $t>0$ and $k=(k_{i,j})$, then
\begin{equation*}
\cos(\v{g e_{n+1}}{N})=k_{1,1}. 
\end{equation*}
Furthermore, we will find it useful to fix the matrix 
\begin{equation*}
g_N:=  \left( \begin{array}{ccc}
\sqrt{2} & &1 \\
&I_{n-1}&\\
1& &\sqrt{2}  \end{array} \right)\in G
\end{equation*}
 satisfying $g_Ne_{n+1}=N$. 

It is an exercise in linear algebra to verify the following
(non-unique) decomposition of $K$:

\begin{lemma}\label{MkMdecomplemma}
Let $n>2$. Every $k=(k_{i,j})\in K$ can be written as 
$$k=m_1k^{\theta} m_2,$$ 
where $m_i\in M$ and
\begin{equation*}
k^{\theta}:=\left( \begin{array}{ccc}
\cos{\theta} & \sin{\theta}& \\
-\sin{\theta}&\cos{\theta}&\\
& &I_{n-1}  \end{array} \right)
\end{equation*}
is uniquely determined by $\theta=\v{ka_t}{g_N}$ for any $t>0$. We
have also $\cos \theta=k_{1,1}$.
\end{lemma}

\begin{remark}
Note that in the case $n=2$, we only get the trivial statement
\begin{equation*}
K=\left\{k^{\theta} : \theta\in[-\pi,\pi)\right\}.
\end{equation*}
From now on, whenever we use the decomposition in Lemma \ref{MkMdecomplemma}, we will only give statements and provide calculations for the case $n>2$. However, using the above observation it should always be clear how to change a statement (calculation) in order to arrive at a valid statement (calculation) also when $n=2$.
\end{remark}

Note that if $g\in G$ has Cartan decomposition $g=m_1k^{\theta(g)}m_2a_{t(g)}m_3k^{\varphi(g)}m_4$
with $m_i\in M$, $t(g)>0$ and $\theta(g),\varphi(g)\in[0,\pi]$, then 
\begin{equation*}
g^{-1}=\tilde m_1k^{\pi -\varphi(g)}\tilde m_2a_{t(g)}\tilde m_3k^{\pi
-\theta(g)}\tilde m_4
\end{equation*} 
for some $\tilde m_i\in M$. 

The following elementary properties of $\v{\cdot}{\cdot}$ are very useful, yet straightforward to verify:

\begin{proposition}\label{angle-props}
Let $n>2$ and let $g,g'\in G\setminus K$. Then the following hold:
\begin{enumerate}[(i)]
\item $\v{kg}{kg'}=\v{g}{g'}$, for every $k\in K$,
\item $\v{gk}{g'}=\v{g}{g'}$, for every $k\in K$,
\item $\v{g}{a_t}=\v{g}{g_N}$ for every $t> 0$, 
\item $\v{g}{a_{-t}}=\pi - \v{g}{a_{t}}$ for every $t> 0$, 
\item $\v{g}{g'}=\v{g'}{g}$,
\item \label{trig-ineq}$\v{g}{g'}\leq \v{g}{g''}+\v{g''}{g'}$ for every $g''\in G\setminus K$,
\item $\v{g}{g_N}=\cos^{-1}(g_{1,n+1} /\sinh t(g))$,
\item $\v{g}{g_N}=\theta(g)$ and $\v{g^{-1}}{g_N}=\pi-\varphi(g)$,
\item $\abs{\theta(g)-\theta(g')}\leq \v{g}{g'}$,
\item $\abs{\varphi(g)-\varphi(g')}\leq \v{g^{-1}}{g'^{-1}}$.
\end{enumerate}
\end{proposition}

We will find it convenient to define the angle $\v{g}{g'}$ also when at least one of $g,g'\in K$. We let
\begin{align*}
\v{g}{g'}:=
\begin{cases}
\v{g_N}{g'}& \text{if $g\in K$ but $g'\notin K$},\\
\v {g}{g_N}& \text{if $g'\in K$ but $g\notin K$},\\
0 &\text{if $g,g'\in K$}.
\end{cases}
\end{align*}
In other words, if one (both) of the elements $g$ and $g'$ occurring in the expression $\v{g}{g'}$ belongs to $K$, then we exchange that element (those elements) with $g_N$ in order to interpret the angle $\v{g}{g'}$. We admit that the above extension is rather arbitrary. However, we note that our choice is natural in the sense that Proposition \ref{angle-props} will continue to hold also for this extended concept of angles. 

\subsection{Normalization of Haar measure and integration formulas}\label{Haar-section}

We normalize the Haar measure $dk$ on $K$ so that 
$$\text{vol}(K)=\int_Kdk=1.$$ 
Furthermore, recalling the identification $G/K\simeq \H^n$, we normalize the Haar measure $dg$ on $G$ in such a way that the induced measure on $G/K$ corresponds to the standard (hyperbolic) measure $d\mu_{\H^n}$ on $\H^n$. In particular, for any cofinite $\Gamma\subset G$ and any (nice) fundamental domain $F_{\Gamma}$ of $\Gamma$, we find that $\mu_{\H^n}(F_{\G})=\text{vol}\left(\G\backslash G\right)$.

With these normalizations, we get the following integration formula (see e.g.\ \cite[Prop.\ 1.17 (p.\ 381)]{Helgason:1962}):

\begin{proposition}\label{KAKmeasure}
Let the Haar measures $dk$ on $K$ and $dg$ on $G$ be normalized as above. Then, for any function $f\in L^1(G)$, 
we have
\begin{align*}
\int_Gf(g)\,dg=\omega_n\int_K\int_0^{\infty}\int_K f(k_1a_tk_2)(\sinh t)^{n-1}\,dk_1dtdk_2,
\end{align*}
where $\omega_n$ denotes the $(n-1)$-dimensional volume of the unit sphere $S^{n-1}\subset\R^n$. 
\end{proposition}

We will also need the following closely related formula:

\begin{proposition}\label{MkMmeasure}
Let $n>2$. Let the Haar measure $dk$ on $K$ be normalized as above and let the Haar measure $dm$ on $M$ be normalized so that $\text{vol}(M)=\omega_n^{-1}$. Then, for any function $f\in L^1(K)$, we have
\begin{align*}
\int_Kf(k)\,dk=\omega_n\omega_{n-1}\int_M\int_0^{\pi}\int_M f\left(m_1k^{\theta}m_2\right)(\sin\theta)^{n-2}\,dm_1d\theta dm_2,
\end{align*}
where $\omega_j$ denotes the $(j-1)$-dimensional volume of the unit sphere $S^{j-1}\subset\R^j$. 
\end{proposition}

\begin{proof}
To begin, we recall that $K/M\simeq S^{n-1}$. Furthermore, using the explicit form of the volume element of the $(n-1)$-sphere in terms of spherical coordinates, we easily find a relation of measures which in integrated form becomes 
\begin{equation*}
\omega_n=\omega_{n-1}\int_0^{\pi}(\sin\theta)^{n-2}\,d\theta.
\end{equation*}
We conclude that 
 $dk$ can be written as a positive multiple of
 $(\sin\theta)^{n-2}\,dm_1d\theta dm_2$. Finally, integrating the
 constant function $f(k)\equiv1$, we find that  the positive multiple
 equals $\omega_n\omega_{n-1}$.
\end{proof}

We note that if $g=k_g a_{t(g)}k_g'$, then $\norm{g}^2=2\cosh t(g)$,
and  with the above normalizations
\begin{equation}\label{BQvolume}
\text{vol}(B_Q)\sim \frac{\omega_n}{2^{n-1}(n-1)}Q^{2(n-1)}\qquad\text{as }Q\to\infty.
\end{equation} 
Indeed, this fact follows from Proposition \ref{KAKmeasure}, since we readily get
\begin{multline*}
\text{vol}(B_Q)=\omega_n\int_0^{\cosh^{-1}(Q^2/2)}(\sinh t)^{n-1}\,dt\\
\sim\frac{\omega_n}{2^{n-1}}\int_0^{2\log Q}e^{(n-1)t}\,dt\sim\frac{\omega_n}{2^{n-1}(n-1)}Q^{2(n-1)}
\end{multline*} 
as $Q\to\infty$.

\begin{remark}\label{LPremark}
In relation to the asymptotic formula \eqref{BQvolume}, let us recall
that Lax and Phillips \cite[Thm.\ 1]{LaxPhillips:1982a} (see also
\cite{Levitan:1987a, ElstrodtGrunewaldMennicke:1988a}) proved that, for cofinite $\Gamma\subset G$, we have 
\begin{align*}
\#N_\G(Q)\sim\frac{\text{vol}(B_Q)}{\text{vol}\left(\G\backslash G\right)}\sim \frac{\omega_n}{2^{n-1}(n-1)\text{vol}\left(\G\backslash G\right)}Q^{2(n-1)}
\end{align*}
as $Q\to\infty$ (cf.\ \eqref{LP-intro}).
\end{remark}

\subsection{The pair correlation counting function}\label{heuristics} 

We now give a brief discussion of the normalized counting function
\begin{equation}\label{pair-correlation-function}
R_{2,Q}(\xi)=\frac1{\#N_\G(Q)}\#\left\{\g,\g'\in N_{\G}(Q) : \gamma^{-1}\gamma'\notin K, \v{\g}{\g'}< \frac{2k_{n,\G}}{Q^{2}}\xi\right\}.
\end{equation}
Our purpose is to give a short heuristic determination of the value of
$k_{n,\G}$ that -- if equidistribution were uniform in all parameters
-- would make $R_{2,Q}(\xi)$ tend to $\xi^{n-1}$ as $Q\to\infty$. 

We recall (see, e.g., \cite[Thm.\ 2]{Nicholls:1983a}) that angles in hyperbolic lattices are
equidistributed in the sense that, for every fixed $g\in G$ and angle $\theta\in[0,\pi]$, we have 
\begin{equation*}\label{equidistribution}
\frac{\#\{\g\in N_{\G}(Q) : \v{\g}{g}< \theta \}}{\#N_\G(Q)}\to\frac{\vol{S_\theta^{n-1}}}{\vol{S^{n-1}}},\quad \textrm{ as }Q\to\infty.\end{equation*}
Here $S_\theta^{n-1}:=\{x\in S^{n-1}\subseteq \R^n: x\cdot e_1> \cos
\theta\}$ is a spherical cap of opening angle
$\theta$. Hence, using Remark \ref{LPremark} and spherical coordinates, we get
\begin{align*}
\#\{\g\in N_{\G}(Q)\setminus\{g\} : \v{\g}{g}< \theta \}&\sim\#\{\g\in N_{\G}(Q) : \v{\g}{g}< \theta \}\\
&\sim\frac{\vol{S_\theta^{n-1}}}{\vol{S^{n-1}}}\#N_\G(Q)\\
&\sim\frac{V_{n-1}}{\text{vol}\left(\G\backslash G\right)}\int_0^{\theta}(\sin t)^{n-2}\,dt\left(\frac{Q^2}{2}\right)^{n-1}
\end{align*}
as $Q\to\infty$, where $V_n=\omega_n/n$ denotes the volume of the unit ball in $\R^n$. Since we are interested in small values of $\theta$ (in fact, in our case, $\theta\to0$ as $Q\to\infty$), we furthermore note that 
\begin{align*}
\int_0^{\theta}(\sin t)^{n-2}\,dt\sim\frac{\theta^{n-1}}{n-1}
\end{align*}
as $\theta\to 0$.
Thus, combining these two asymptotic formulas, we expect 
\begin{align}\label{BQTHETAASYMP}
\#\{\g\in N_{\G}(Q)\setminus\{g\} : \v{\g}{g}< \theta \}\approx\frac{V_{n-1}}{(n-1)\text{vol}\left(\G\backslash G\right)}\left(\frac{Q^2\theta}{2}\right)^{n-1}
\end{align}
as $Q\to\infty$ and $\theta\to 0$.

Recall that we are interested in understanding the function
$R_{2,Q}(\xi)$ in \eqref{pair-correlation-function}. In particular, we
are interested in finding the value of $\theta$ such that the right hand
side in \eqref{BQTHETAASYMP}, on average over $g=\gamma'\in N_\G(Q)$
and as $Q$ tends to infinity, is expected to be close to $\xi^{n-1}$. Thus, we solve the equation
\begin{align*}
\frac{V_{n-1}}{(n-1)\text{vol}\left(\G\backslash G\right)}\left(\frac{Q^2\theta}{2}\right)^{n-1}=\xi^{n-1},
\end{align*}
and we immediately find the solution
\begin{align*}
\theta=\frac{2\xi}{Q^2}\left(\frac{(n-1)\text{vol}\left(\G\backslash G\right)}{V_{n-1}}\right)^{\frac1{n-1}}.
\end{align*}
Based on the above discussion, we define
\begin{equation}\label{kn-constant}
k_{n,\G}:=\left(\frac{(n-1)\text{vol}\left(\G\backslash G\right)}{V_{n-1}}\right)^{\frac1{n-1}};
\end{equation}
this is the constant that we will use throughout to normalize the
angles in the counting function $R_{2,Q}(\xi)$ (see \eqref{pair-correlation-function}).

\subsection{Decay of matrix coefficients}

A basic ingredient in our argument is the decay properties of matrix coefficients of the
right regular representation of $G$ on the space  $L^2(\Gamma\backslash
G)=L^2(\Gamma\backslash G,dg)$. 
That is, for $\Phi_1,\Phi_2\in L^2(\Gamma\backslash G)$, we are interested in the properties of 
\begin{equation*}
  g\mapsto \inprod{\pi(g)\Phi_1}{\Phi_2}_{\G\backslash G},
\end{equation*}
where $\inprod{f_1}{f_2}_{\G\backslash G}=\int_{\G\backslash G}f_1(g)\overline{f_2(g)}\,dg$ and $(\pi(g)f)(h)=f(hg)$. We refer to
\cite{Oh:2014} and the references therein for recent developments on
various counting problems using 
precise information on the decay of matrix coefficients. 

Since we assume $\Gamma$ to be a lattice in $G$, we have $L^2(\Gamma\backslash
G)=\C\oplus L_0^2(\Gamma\backslash G)$. It is well-known that there exists $s_0$ in $((n-1)/2,n-1)$ such that
$L_0^2(\G\backslash G)$ does not contain any complementary
series representation with parameter $s\geq s_0$. Recall that this is equivalent to the statement that the non-trivial
spectrum of the automorphic Laplacian is contained in the interval $(s_0(n-1-s_0),\infty)$, i.e.\ to a
spectral gap. Furthermore, for $\Phi_1,\Phi_2\in L_0^2(\Gamma\backslash G)$ smooth and $K$-finite, 
we have 
\begin{equation*}
  \big|\inprod{\pi(g)\Phi_1}{\Phi_2}_{\G\backslash G}\big|\leq
  Ce^{(s_0-n+1)t(g)}\prod_{i=1,2}(\dim(\pi(K)\Phi_i))^{1/2}\|\Phi_i\|_{L^2}
\end{equation*}
for some positive constant $C$ depending only on $s_0$ (see e.g.\ \cite[Eq.\ (5.4)]{KontorovichOh:2011} for the case $n=3$; the argument given there
readily extends to general $n\geq2$). Using Fourier decomposition on the compact group $K$, it is possible to
remove the assumption of $K$-finiteness at the expense of a Sobolev norm. For $\Phi\in C^\infty(\G\backslash G)$ and $l\in\N$, we define the $l$th Sobolev norm of $\Phi$ by
\begin{equation}\label{Sobolevdefinition}
\mathcal{S}_l(\Phi):=\sum\|X(\Phi)\|_{L^2},
\end{equation}
where the sum is taken over all monomials of degree at most $l$ in
some fixed basis for the Lie algebra of $G$. The above argument leads
to the following theorem (see e.g.\ \cite[Thm. 3.1]{Oh:2014} and the
references therein):

\begin{theorem} \label{main-bound} 
There exist $\frac{n-1}{2}<s_0< (n-1)$ and $l\in \N$ such that for all $\Phi_1,\Phi_2\in L^2(\G\backslash G)\cap C^\infty(\G\backslash G)$, we have 
\begin{equation}\label{matrixcoefficientsestimatewithsobolevnorms}
\langle\pi(g)\Phi_1,\Phi_2\rangle_{\G\backslash G}=\frac{\langle\Phi_1,1\rangle_{\G\backslash G}\langle1,\Phi_2\rangle_{\G\backslash G}}{\vol{\G \backslash G}}+O\left(e^{(s_0-n+1)t(g)}\mathcal{S}_l(\Phi_1)\mathcal{S}_l(\Phi_2)\right).
\end{equation}
\end{theorem}

\begin{remark}\label{Kfiniteremark}
We note that the argument removing the $K$-finiteness above is done independently for the two functions $\Phi_1$ and $\Phi_2$. In particular, this implies that in the case where $\Phi_i$ ($i=1$ or $2$) is $K$-invariant, we can replace the Sobolev norm $\mathcal{S}_l(\Phi_i)$ in \eqref{matrixcoefficientsestimatewithsobolevnorms} by the corresponding $L^2$-norm $\|\Phi_i\|_{L^2}$. In addition we mention that, in the case $n=2$ where it is possible to take $l=1$, Venkatesh \cite[Sect.\ 9.1.2]{Venkatesh:2010} has given an interpolation argument that replaces the Sobolev norm $\mathcal{S}_1(\Phi_i)$ in \eqref{matrixcoefficientsestimatewithsobolevnorms} by $ \mathcal{S}_1(\Phi_i)^{1/2+\epsilon}\|\Phi_i\|_{L^2}^{1/2-\epsilon}$ for any $0<\epsilon<\frac12$.
\end{remark}

\begin{remark}
The implied constant in \eqref{matrixcoefficientsestimatewithsobolevnorms} depends on $s_0$. However, since we will work with a fixed admissible $s_0$, we choose not to indicate this dependence. 
\end{remark}

\section{Stability }\label{stability}
We will need some information about how $\norm{g}$ and $\theta(g)$ change
when we multiply from the right with $M\in G$. This is adressed in the
following proposition:

\begin{proposition}\label{change-under-right-mult}
  Let $g,M\in G$. Then 
  \begin{equation}\label{norm-change}
\norm{gM}^2=2(\cosh t(g)\cosh t(M)+\cos v\sinh t(g)\sinh t(M)).
\end{equation}
Assume further that $t(g)>t(M)$. Then  $\v{gM}{g}<\pi/2$ and
\begin{equation}\label{angle-change}
\tan \v{gM}{g}=\frac{\sin v\sinh t(M) }{\cosh t(M)\sinh t(g)+\cos
  v\cosh t(g)\sinh t(M)}.
\end{equation}
In both cases $v=\pi -\v{g^{-1}}{M}$.
\end{proposition}

\begin{proof}
For $g\in K$ or $M\in K$ everything is clear. Assume this not to be
  the case. Using the Cartan decompositions of $g$ and $M$, and that $d(x,y)$
  is a point-pair invariant, we see that
  \begin{equation*}
    \norm{gM}^2=2\cosh d(a_{t(g)}k_g'k_Ma_{t(M)}e_{n+1},e_{n+1}).
  \end{equation*}
 It follows from the definition of $d$ that
 $\cosh d(a_{t(g)}k_g'k_Ma_{t(M)}e_{n+1},e_{n+1})$ equals the lower entry of
 $a_{t(g)}k_g'k_Ma_{t(M)}e_{n+1}$, from which \eqref{norm-change}
 follows easily by inspection if $\cos v$ equals the
 upper left entry of  $k_g'k_M$. However,  the upper left entry of $k_g'k_M$ equals
 $\cos(\v{k_g'k_Ma_t}{g_N})$ for any $t>0$, and by Proposition
 \ref{angle-props} we see that 
 \begin{align*}
   \v{k_g'k_Ma_{t(M)}}{g_N}=&\v{k_g'k_Ma_{t(M)}}{a_{t(g)}}\\=&\pi-\v{k_Ma_{t(M)}}{k_g'^{-1}a_{-t(g)}}=\pi-\v{M}{g^{-1}}.
 \end{align*}

To prove \eqref{angle-change}, we note that
\begin{equation*}
\v{gM}{g}=\v{a_{t(g)}k_g'k_Ma_{t(M)}}{a_{t(g)}}=\v{a_{t(g)}k_g'k_Ma_{t(M)}}{g_N}.
\end{equation*}
Using that the upper left entry of $k_g'k_M$ equals $\cos
v$,  and that, by orthogonality, the sum
of the squares of the elements in the rest of the first column in $k_g'k_M$ equals
$\sin^2v$,  a direct computation from the definition of $v$ shows that 
\begin{align}\label{willi}
 \nonumber \cos&\,\v{gM}{g}=\cos \v{a_{t(g)}k_g'k_Ma_{t(M)}}{g_N}\\ 
&=\frac{\sinh t(g)\cosh t(M) +\cos v \cosh t(g) \sinh
  t(M)}{\sqrt{(\sinh t(g)\cosh t(M)+\cos v \cosh t(g) \sinh t(M))^2+\sin^2 v\sinh^2t(M)}}.
\end{align}
We note that $t(g)>t(M)$ implies
that the numerator of \eqref{willi} is positive, so the angle
$\v{gM}{g}$ is at most $\pi/2$. From this the result follows easily.
\end{proof}

For any (small) $\delta>0$, we define
\begin{equation*}
  A_\delta:=\{a_t: \abs{t}\leq \delta\}.
\end{equation*}
Then clearly 
\begin{equation}\label{Bdelta}
B_{\delta_1}=KA_\delta K,
\end{equation}
where $\delta_1^2=2\cosh \delta$.

\begin{lemma}\label{first-bounds}
  Let $g\in G$ with $\norm{g}>3$, and let $h\in B_{\delta_1}$. Then,
  for $\delta>0$ sufficiently small, we have 
  \begin{enumerate}[(i)]
  \item \label{bound-one} $\norm{gh}=\norm{g}(1+O(\delta))$,
  \item \label{bound-two}$\v{g}{gh}=O\left(\frac{\delta}{\norm{g}^2}\right)$.
  \end{enumerate}
\end{lemma}

\begin{proof}
 From Proposition \ref{change-under-right-mult},
 we see that 
 \begin{equation*}
\norm{gh}^2=2\cosh(t(g))(1+O(\delta^2))+O(\delta\norm{g}^2)=\norm{g}^2(1+O(\delta)),
\end{equation*}
which implies \eqref{bound-one}.

Similarly, from Proposition \ref{change-under-right-mult}, we see that for $\delta$ sufficiently small
\begin{align*}
  \abs{\tan \v{g}{gh}}
=O\left(\frac{\delta}{\norm{g}^2(\tanh t(g)\cosh t(h)-\sinh
  t(h))}\right)=O\left(\frac{\delta}{\norm{g}^2}\right),
\end{align*}
which implies \eqref{bound-two} since $\abs{v}\leq \abs{\tan v}$ for
$\abs{v}\leq \pi/2.$
\end{proof}

\begin{remark}\label{inversions-too}
  By using that $t(g^{-1})=t(g)$,  Lemma \ref{first-bounds}
  implies $\norm{hg}=\norm{g}(1+O(\delta))$ and
  $\v{g^{-1}}{g^{-1}h^{-1}}=O\left(\frac{\delta}{\norm{g}^2}\right)$. We note
    also that combining the above angle bounds with Proposition
    \ref{angle-props}, we obtain
    \begin{equation*}
      \abs{\theta(g)-\theta(gh)}=O\left(\frac{\delta}{\norm{g}^2}\right),
        \quad \abs{\varphi(g)-\varphi(hg)}=O\left(\frac{\delta}{\norm{g}^2}\right)
    \end{equation*}
for $\norm{g}>3$ and $h\in B_{\delta_1}$.
\end{remark}

We now define
\begin{equation}
 \label{Ddelta} D_\delta:=K_\delta A_\delta K_\delta,
\end{equation}
where $K_\delta$ is defined by
\begin{equation}\label{Kdelta}
K_\delta:=\{k\in K:\abs{ka-a} <\delta, \textrm{ for
  all }a\in S^{n}\subseteq \R^{n+1}\}.
\end{equation}
Note that the elements of $K_\delta$ rotate any given direction in
$\R^{n+1}$ by at most a small amount.

\begin{lemma}\label{fattening-stability}
 For $\delta>0$ sufficiently small the following holds: Let $g\in G$, with $\norm{g}^2>3$ , and let $g_1=h_1gh_2\in
  B_{\delta_1}gD_\delta$. Then 
  \begin{enumerate}[(i)]
\item\label{t-growth}    $t(g_1)-t(g)=O(\delta)$,
\item\label{v-growth} $\v{g_1^{-1}}{g^{-1}}=O(\delta).$
  \end{enumerate}
\end{lemma}

\begin{proof}
To prove \eqref{t-growth}, we use Lemma \ref{first-bounds}, Remark \ref{inversions-too} and the mean value theorem.  For $\delta>0$ sufficiently small, $\norm{g_1}^2/2$ and $\norm{g}^2/2$ are
larger than say $5/4$. Also, $r\leq 2\sqrt{r^2-1}$ for $r>5/4$. Hence, 
for some $r$ between $\norm{g_1}^2/2$ and $\norm{g}^2/2$, we obtain
\begin{align*}
  \abs{t(g_1)-t(g)}=&\abs{\cosh^{-1}\left(\norm{g_1}^2/2\right)-\cosh^{-1}\left(\norm{g}^2/2\right)}\\
=&\frac{1}{\sqrt{r^2-1}}\abs{\norm{g_1}^2/2-\norm{g}^2/2}
=O\bigg(\frac{\delta\norm{g}^2}{r}\bigg)
=O(\delta).
\end{align*}

To prove \eqref{v-growth}, we observe that by Lemma \ref{first-bounds}, Remark \ref{inversions-too}
and Proposition \ref{angle-props} we have, since $h_1^{-1},h_2^{-1}\in B_{\delta_1}$, that 
\begin{align*}
0\leq  &\v{g_1^{-1}}{g^{-1}}
\leq\v{h_2^{-1}g^{-1}h_1^{-1}}{h_2^{-1}g^{-1}}+\v{h_2^{-1}g^{-1}}{g^{-1}}\\
=&O\bigg(\frac{\delta}{\norm{h_2^{-1}g^{-1}}^2}\bigg)+\v{h_2^{-1}g^{-1}}{g^{-1}}
=O(\delta)+\v{h_2^{-1}g^{-1}}{g^{-1}}.
\end{align*}

Let $C>0$ be a fixed constant. We claim that, for any $h\in G$ satisfying
\begin{equation}\label{another-condition}
\max_{i,j}\abs{(h-I)_{ij}}\leq
  C\delta,
\end{equation} 
we have -- for $g\in G$ with $\norm{g}$
bounded away from $1$ -- that
\begin{equation}\label{stronger-statement}
  \v{hg}{g}=O(\delta).
\end{equation}
(Note that, for $\delta$ sufficiently small, we have $hg,g\not \in
K$.) 
From this and the above considerations, \eqref{v-growth}  follows
directly since $h_2^{-1}$ satisfies \eqref{another-condition}.

To prove \eqref{stronger-statement}, we first note that for any angle
$0\leq v \leq \pi$ we have\footnote{Note that the right-hand side of \eqref{basic-inequality} is twice the Euclidean distance between $(\cos
v,\sin v)$ and $(1,0)$.}
\begin{equation}
  \label{basic-inequality}
  v\leq 2\sqrt{2(1-\cos v)}.
\end{equation}
For $g=k_ga_t k'_g$ we have, by Proposition
\ref{angle-props}, that 
\begin{equation*}
  \v{hg}{g}
  =\v{h_ga_t}{a_t},
\end{equation*} 
where $h_g=k_g^{-1}hk_g$. Note that
\begin{equation}\label{slightly-worse}\max_{i,j}\abs{(h_g-I)_{ij}}=O(\delta).
\end{equation} 
Let $u:\R^{n+1}\to\R^n$ be the
mapping which drops the last coordinate in the standard basis
representation. From \eqref{slightly-worse} follows, using equivalence
of norms on finite dimensional vector spaces, that 
\begin{equation}\label{bound-again}
  \abs{u(h_{g,1})-u(e_1)}=O(\delta), \quad \abs{u(h_{g,n+1})}=O(\delta),
\end{equation}
where $h_{g,i}$ is the $i$th column of $h_g$. 

Now, by definition $\cos\v{hg}{g}$ is the first entry
in $u(h_ga_te_{n+1})/\abs{u(h_ga_te_{n+1})}$.
Since this vector has
norm $1$, the sum of squares of the remaining entries equals
$\sin^2\v{hg}{g}$. It follows that 
\begin{equation}\label{abs=sqrt}
  \abs{\frac{u(h_ga_te_{n+1})}{\abs{u(h_ga_te_{n+1})}}-u(e_1)}=\sqrt{2(1-\cos\v{hg}{g})}.
\end{equation}
We note that, since $a_te_{n+1}=\sinh(t)  \cdot e_1+ \cosh(t) \cdot
e_{n+1}$, we have 
\begin{equation*}
u(h_ga_te_{n+1})=\sinh(t)\cdot
  u(h_{g,1})+\cosh(t)\cdot u(h_{g,n+1}).
\end{equation*}
Therefore, using \eqref{bound-again} and the fact that $t$ is bounded away from
zero, we have
  \begin{equation} \label{abs-estimation}
    \abs{u(h_ga_te_{n+1})}
    =\sinh t(1+O(\delta)).
  \end{equation}
Finally, combining \eqref{basic-inequality}, \eqref{abs=sqrt}, \eqref{abs-estimation}, and \eqref{bound-again} we find,
after a small computation that
\begin{equation*}
  \v{hg}{g}\leq 2\sqrt{2(1-\cos\v{hg}{g})}=2\abs{\frac{u(h_ga_te_{n+1})}{\abs{u(h_ga_te_{n+1})}}-u(e_1)}=O(\delta),
\end{equation*}
which proves the claim.
\end{proof}

\begin{lemma}\label{stability-left-right} For any $\delta>0$ sufficiently small the following holds:
  Let $g,M\in G$, $M\not \in K$. Assume $\norm{g}\geq
  10\norm{M}$. Then, for every $g_1\in B_{\delta_1} g D_\delta$, we have
  $\v{g_1}{g_1M}$, $\v{g}{gM}<\pi/2$  and
  \begin{enumerate}[(i)]
  \item \label{uno} $\norm{g_1M}^2=\norm{gM}^2+O(\delta\norm{g}^2\norm{M}^2)$,
 \item \label{dos}$\tan \v{g_1}{g_1M}=\tan \v{g}{gM}+O\left(\frac{\delta\norm{M}^4}{\norm{g}^2}\right)$.
  \end{enumerate}
\end{lemma}

\begin{proof}
  The proof follows closely the proof of \cite[Lemma
  2.20]{KelmerKontorovich:2013}. 
  Consider the functions (denoted by  $F_1,F_2$ in \cite{KelmerKontorovich:2013}) 
  \begin{align*}
    G_1(t,v)&:=2(\cosh t(M)\cosh t + \cos(\pi -v)\sinh t(M)\sinh t),\\
G_2(t,v)&:=\frac{\sin (\pi -v)\sinh t(M)}{\cosh t(M)\sinh t + \cos(\pi - v)\cosh t\sinh t(M)}.
  \end{align*}
The assumption $\norm{g}\geq 10\norm{M}$ implies that $t(g)>t(M)+1$.
Therefore, using Lemma \ref{fattening-stability}, we find that $t(g_1),t(g)>t(M)$, which by Proposition \ref{change-under-right-mult}
implies that $\v{g_1}{g_1M}$, $\v{g}{gM}<\pi/2$.

 By Proposition \ref{change-under-right-mult}, we have 
\begin{align*}
  \norm{g_1M}^2-\norm{gM}^2&=G_1(t(g_1), \v{g_1^{-1}}{M})-G_1(t(g),
  \v{g^{-1}}{M}),\\
\tan\v{g_1}{g_1M}-\tan\v{g}{gM}&=G_2(t(g_1), \v{g_1^{-1}}{M})-G_2(t(g), \v{g^{-1}}{M}).
\end{align*}
Furthermore, using Lemma \ref{fattening-stability} and Proposition \ref{angle-props}, we find 
\begin{equation*}\abs{t(g_1)-t(g)},
\abs{\v{g_1^{-1}}{M}-\v{g^{-1}}{M}}=O(\delta).
\end{equation*} 
Finally, using the mean value
theorem and the Cauchy-Schwarz inequality, the result follows from the following
bounds on the partial derivatives of $G_1$ and $G_2$ valid when $t>t(M)+1$:
\begin{align*}
\frac{\partial G_1}{\partial t}(t,v), \frac{\partial
    G_1}{\partial v}(t,v)&=O(\norm{M}^2\cosh t),\\
     \quad \frac{\partial G_2}{\partial t}(t,v), \frac{\partial
    G_2}{\partial v}(t,v)&=O(\norm{M}^4/\cosh t) 
\end{align*}
(see \cite[Lemma
  2.20]{KelmerKontorovich:2013} for the proofs of these estimates). 
\end{proof}

\section{Volume computations}\label{vols}

In this section we consider the problem of asymptotically determining the volume of the set
\begin{equation*}
\mathcal R_M(Q,\xi)=\left\{g\in B_Q : \norm{gM}\leq Q,\v{g}{gM}<\frac{2\xi}{Q^2}\right\}
\end{equation*}
for $M\in\G$, $M\notin K$. In order to formulate our results we introduce some more notation. We let
\begin{equation*}
A:=A(l)=\cosh l,\qquad B:=B(l)=\sinh l,\qquad C:=C(l)=2\sinh(l/2).
\end{equation*}
For $\xi\leq B$, we define the real numbers
\begin{align}
  \alpha&:=\alpha(\xi,l)=\sqrt{1-\frac{\xi^2}{B^2}},\label{alpha}\\
\lambda_\pm&:=\lambda_\pm(\xi,l)=\frac{-\xi^2 \frac A B\pm\sqrt{1-\frac{\xi^2}{B^2}}}{\xi^2+1}.\label{lambda}
\end{align}
Furthermore, we introduce the set
\begin{equation}\label{set-I}
  I(\xi,l):=\begin{cases}
[-1,\lambda_-)\cup(\alpha, 1]& \textrm {if $\xi\leq C$},\\
[-1,\lambda_-)\cup(\lambda_+, -\alpha)\cup(\alpha,1]& \textrm {if $C<\xi\leq B$},\\
[-1,1]& \textrm {if $B< \xi$}.
\end{cases}
\end{equation}
In situations where the variable $l$ is temporarily fixed to equal $l=t(M)$ for some element $M\in\G$, we will also find it convenient to use the notation
\begin{align*}
A_M=A(t(M)),\quad B_M=B(t(M)),\quad C_M=C(t(M)),
\end{align*}
and
\begin{align}\label{subindex_M}
\alpha_M(\xi)=\alpha(\xi,t(M)),\quad \lambda_{\pm,M}(\xi)=\lambda_\pm(\xi,t(M)).
\end{align}
The function
\begin{align}\label{DEFOFF}
f_{\xi}(l)=\frac1{\xi^n}\int_{I(\xi,l)}\left(1-y^2\right)^{n-2}\left(y+\coth l\right)^{-(n-1)}\,dy
\end{align}
(where $I(\xi,l)$ is defined by \eqref{set-I}) will play a central role in the rest of the
paper. It enters our discussion as part of the following result:

\begin{theorem}\label{volume-theorem}
For every $M\in\G$, $M\notin K$, and for $\xi/Q^2$ sufficiently small, we have
\begin{align}\label{VOLUMETHM}
\textup{vol}\left(\mathcal R_M\left(Q,\xi\right)\right)=\frac{\omega_{n-1}Q^{2(n-1)}}{2^{n-1}}\int_0^{\xi}f_{\zeta}\left(t(M)\right)\,d\zeta+O\left(g(\xi)\|M\|^{2(n-1)}Q^{2\frac{(n-1)^2}{n+1}}\right)
\end{align}
as $Q\to\infty$, where
\begin{align}\label{gdef}
g(\xi)=\xi^{n-1}+\xi^{-(n-1)}.
\end{align}
\end{theorem}

\begin{remark}\label{M-remark}
Note that we need $\|M\|=o\left(Q^{2/(n+1)}\right)$ for the error term in \eqref{VOLUMETHM} to be non-trivial. 
\end{remark}

\begin{remark}\label{n=2,3}
For any given value of $n$, we can perform the integration in \eqref{DEFOFF} and obtain completely explicit expressions for the function $f_\xi$. When $n=2$ and $n=3$, we have
\begin{equation*}
f_{\xi}(l)=\frac{2}{\xi^2}\times
\begin{cases}
l-\log\big(A+\sqrt{B^2-\xi^2}\big)& \textrm {if $\xi\leq C$},\\
l+\log(1+\xi^2)-2\log\big(A+\sqrt{B^2-\xi^2}\big)& \textrm {if $C<\xi\leq B$},\\
l&  \textrm {if $B< \xi$},
\end{cases}
\end{equation*}
and
\begin{equation*}
f_{\xi}(l)=\frac{4}{\xi^3}\times
\begin{cases}
\frac{A}{B}\log\Big(\frac{A+B}{A+\sqrt{B^2-\xi^2}}\Big)+\frac{(\xi^2+2)\sqrt{B^2-\xi^2}+A\xi^2}{2B(\xi^2+1)}-1& \textrm {if $\xi\leq C$},\\
\frac{A}{B}\log\Big(\frac{(A+B)(\xi^2+1)}{(A+\sqrt{B^2-\xi^2})^2}\Big)+\frac{(\xi^2+2)\sqrt{B^2-\xi^2}}{B(\xi^2+1)}-1& \textrm {if $C<\xi\leq B$},\\
\frac{A}{B}l-1&  \textrm {if $B< \xi$},
\end{cases}
\end{equation*}
respectively. 
When $n$ gets larger such expressions become very cumbersome. Plotting
$f_\xi(l)$ for fixed $\xi$ or $l$ seems  to indicate, however,  that the
\lq complexity\rq{} of $f_\xi(l)$ doesn't grow significantly when $n$ gets larger (cf.\ Figure \ref{n=2} and Figure \ref{n=7}).
\begin{figure}
\centering
\includegraphics{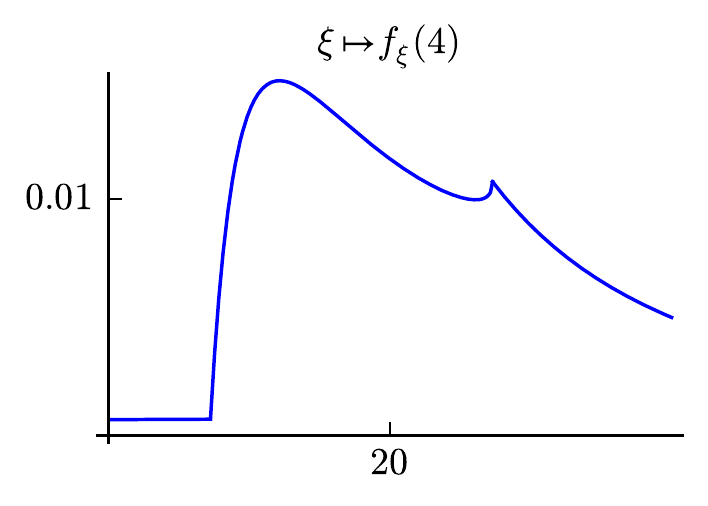}\includegraphics{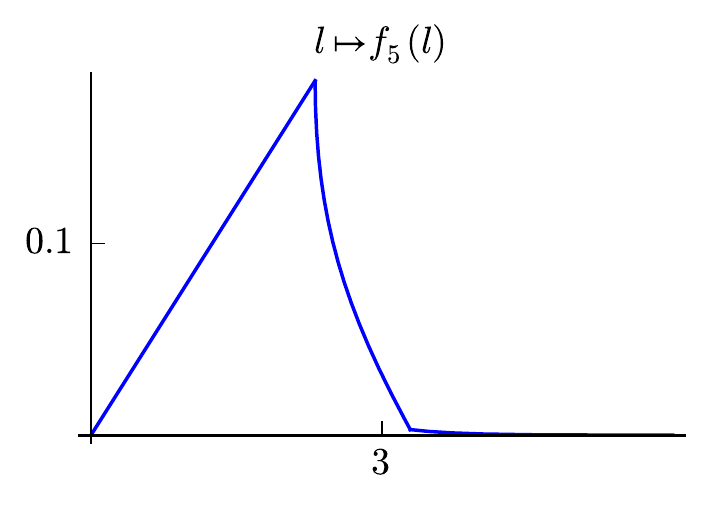}
\caption{Plots of $f_\xi(l)$ with one argument fixed when $n=2$.}\label{n=2}
\end{figure}
\begin{figure}
\centering
\includegraphics{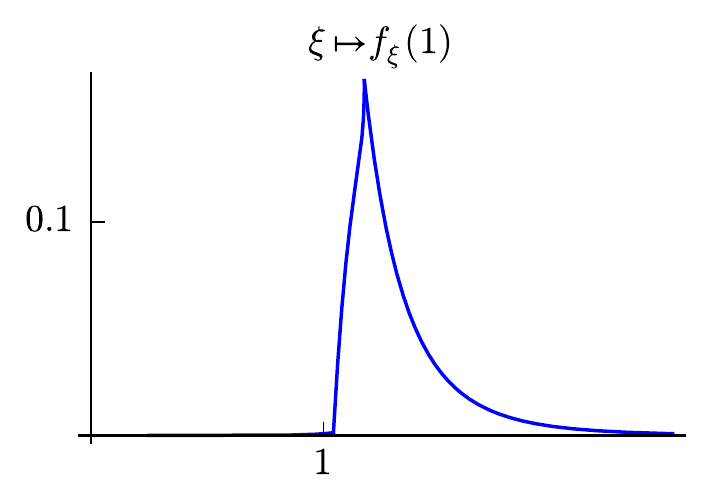}\includegraphics{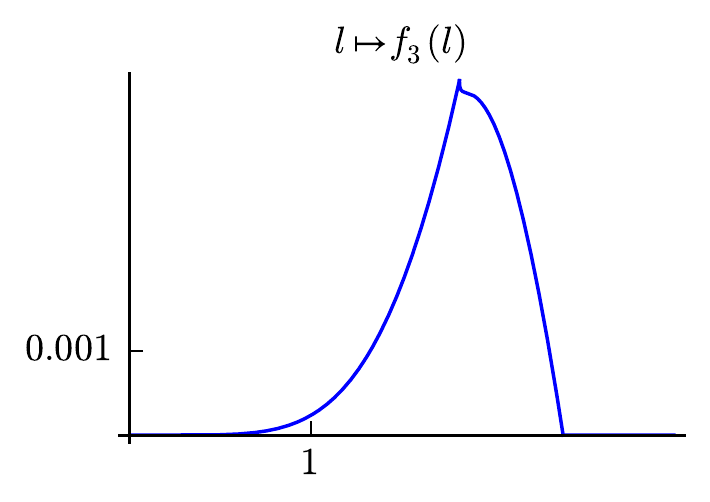}
\caption{Plots of $f_\xi(l)$ with one argument fixed when $n=7$.}\label{n=7}
\end{figure}
\end{remark}

We will prove Theorem \ref{volume-theorem} in two steps. The general strategy of the proof follows closely the one in \cite[Sect. 3]{KelmerKontorovich:2013}. However, since some parts of the proof are computationally different, we have chosen to give a detailed proof for completeness.

\begin{proposition}\label{volume-proposition}
For every $M\in\G$, $M\notin K$, and for $\xi/Q^2$ sufficiently small, we have
\begin{multline}\label{firstvolume}
\textup{vol}\left(\mathcal R_M\left(Q,\xi\right)\right)=\frac{\omega_{n-1}Q^{2(n-1)}}{2^{n-1}}\int_{-1}^1(1-y^2)^{(n-3)/2}\int_{J_{\xi}(y)}x^{n-2}\,dxdy\\
+O\left(g(\xi)\|M\|^{2(n-1)}Q^{2\frac{(n-1)^2}{n+1}}\right)
\end{multline}
as $Q\to\infty$, where $g(\xi)$ is given by \eqref{gdef} and the interval $J_{\xi}(y)$ is defined by 
\begin{equation*}
J_{\xi}(y):=\left\{x\in[0,1] : \frac{B_M\sqrt{1-y^2}}{\xi(A_M+B_My)}\leq x\leq\frac1{A_M+B_My}\right\}.
\end{equation*} 
\end{proposition}

\begin{proof}
To begin, we note that Remark \ref{M-remark} implies that we may assume that
\begin{align}\label{M-remark2}
\|M\|=o\left(Q^{2/(n+1)}\right).
\end{align}
We introduce a large positive parameter $X=X(M,Q)$ satisfying
\begin{equation}\label{Xinterval}
5\norm{M}^2<X<Q^2,
\end{equation}
and use it to truncate the region $\mathcal R_M\left(Q,\xi\right)$ as follows:
\begin{equation*}
\widetilde{\mathcal R}_M\left(Q,\xi\right):=\mathcal R_M\left(Q,\xi\right)\cap\{g\in G : \norm{g}^2>X\}.
\end{equation*}
At the end of this proof (see \eqref{X}), we will determine a value of $X$ that balances the sizes of our error terms. For now, we will confine ourselves to observe that \eqref{BQvolume} immediately implies that 
\begin{equation}\label{truncated volume}
\text{vol}\left(\mathcal R_M\left(Q,\xi\right)\right)=\text{vol}\big(\widetilde{\mathcal R}_M\left(Q,\xi\right)\big)+O(X^{n-1}).
\end{equation}

Next, we describe the conditions determining the set
$\widetilde{\mathcal R}_M\left(Q,\xi\right)$ in terms of the $KA^+K$
coordinates. Writing $g=k_ga_tk_g'$ (recall \eqref{cartan}), we first
note that $\widetilde{\mathcal R}_M\left(Q,\xi\right)$ is left
$K$-invariant, i.e.\ it does not impose any restriction on the first
$K$-component $k_g$. Furthermore, using Proposition
\ref{change-under-right-mult} together with \eqref{Xinterval}, we find
that $\widetilde{\mathcal R}_M\left(Q,\xi\right)$ is determined by the
following inequalities once $\xi/Q^2$ is sufficiently small:
\begin{align}
X<2\cosh t&\leq Q^2,\label{tildeR-in-KAK1}\\
2(A_M\cosh t+B_M\cos v\sinh t)&\leq Q^2,\label{tildeR-in-KAK2}\\
\frac{B_M\sin v}{A_M\sinh t+B_M\cos v\cosh t}&<\tan\left(\frac{2\xi}{Q^2}\right),\label{tildeR-in-KAK3}
\end{align}
where $v=\pi-\v{g^{-1}}{M}=\v{k_g'k_Ma_{t(M)}}{g_N}$. We let $k=k_g'k_M$ and note that $k=m_1k^vm_2$ for some appropriate $m_1,m_2\in M$. We also let $\chi_{\widetilde{\mathcal R}_M\left(Q,\xi\right)}$ denote the indicator function of the set $\widetilde{\mathcal R}_M\left(Q,\xi\right)$. However, since $\widetilde{\mathcal R}_M\left(Q,\xi\right)$ is left $K$-invariant, we will also let $\chi_{\widetilde{\mathcal R}_M\left(Q,\xi\right)}$ denote the indicator function of the set of $(t,k)\in[0,\infty)\times K$ satisfying the conditions \eqref{tildeR-in-KAK1}, \eqref{tildeR-in-KAK2} and \eqref{tildeR-in-KAK3}. Now, using Proposition \ref{KAKmeasure}, we obtain
\begin{align*}
\text{vol}\big(\widetilde{\mathcal R}_M\left(Q,\xi\right)\big)&=\omega_n\int_K\int_0^{\infty}\int_K \chi_{\widetilde{\mathcal R}_M\left(Q,\xi\right)}(k_ga_tk_g')(\sinh t)^{n-1}\,dk_gdtdk_g'\\
&=\omega_n\int_K\int_0^{\infty}\chi_{\widetilde{\mathcal R}_M\left(Q,\xi\right)}(a_tk_g')(\sinh t)^{n-1}\,dtdk_g'.
\end{align*}
Furthermore, changing variables $k_g'\mapsto k$ and applying Proposition \ref{MkMmeasure}, we arrive at\footnote{The calculation in \eqref{tildeR-integral} uses Proposition \ref{MkMmeasure} and hence is valid only for $n>2$. However, exchanging the use of Proposition \ref{MkMmeasure} with an easy symmetry argument, one easily finds that the last line of \eqref{tildeR-integral} is correct also in the case $n=2$ (cf.\ \cite[Prop.\ 3.1]{KelmerKontorovich:2013}).} 
\begin{align}\label{tildeR-integral}
&\text{vol}\big(\widetilde{\mathcal R}_M\left(Q,\xi\right)\big)\\
&=\omega_n^2\omega_{n-1}\int_M\int_0^{\pi}\int_M\int_0^{\infty}\chi_{\widetilde{\mathcal R}_M\left(Q,\xi\right)}(t,m_1k^{v}m_2)(\sin v)^{n-2}(\sinh t)^{n-1}\,dtdm_1dvdm_2\nonumber\\
&=\omega_{n-1}\int_0^{\pi}\int_0^{\infty}\chi_{\widetilde{\mathcal R}_M\left(Q,\xi\right)}(t,k^{v})(\sin v)^{n-2}(\sinh t)^{n-1}\,dtdv.\nonumber
\end{align}

In order to evaluate this integral, we make the following change of variables:
\begin{equation*}
x=\frac{2\cosh t}{Q^2},\qquad y=\cos v.
\end{equation*}
We find that the inequalities \eqref{tildeR-in-KAK1}, \eqref{tildeR-in-KAK2} and \eqref{tildeR-in-KAK3} are transformed into 
\begin{align}
\frac{X}{Q^2}<x&\leq 1,\label{tildeR-in-KAK1'}\\
x(A_M+B_Myz)&\leq 1,\label{tildeR-in-KAK2'}\\
\frac{B_M\sqrt{1-y^2}}{x(A_Mz+B_My)}&<\frac{Q^2}{2}\tan\left(\frac{2\xi}{Q^2}\right),\label{tildeR-in-KAK3''}
\end{align}
where we have used the notation
\begin{equation*}
z:=\tanh t=\sqrt{1-\frac{4}{Q^4x^2}}.
\end{equation*} 

Next, for any $y\in[-1,1]$, we let $\widetilde J_{Q,\xi}(y)$ denote the set of all real numbers $x$ satisfying \eqref{tildeR-in-KAK1'}, \eqref{tildeR-in-KAK2'} and \eqref{tildeR-in-KAK3''}. Then, it follows from \eqref{tildeR-integral} that
\begin{align}\label{the-volume-integral}
\text{vol}\big(\widetilde{\mathcal R}_M\left(Q,\xi\right)\big)&=\frac{\omega_{n-1}Q^2}{2}\int_{-1}^1(1-y^2)^{(n-3)/2}\int_{\widetilde J_{Q,\xi}(y)}\left(\frac{Q^4x^2}{4}-1\right)^{n/2-1}\,dxdy\nonumber\\
&=\frac{\omega_{n-1}Q^{2(n-1)}}{2^{n-1}}\int_{-1}^1(1-y^2)^{(n-3)/2}\int_{\widetilde J_{Q,\xi}(y)}x^{n-2}\,dxdy+\mathcal E(Q),
\end{align}
where
\begin{align*}
\mathcal E(Q)=
\begin{cases}
0& \text{ if $n=2$,}\\
O_\epsilon(Q^{\epsilon})&  \text{ if $n=3$,}\\
O(Q^{2(n-3)})&  \text{ if $n\geq4$.}
\end{cases}
\end{align*}

We continue by noting that
\begin{align*}
\frac{\sqrt{1-y^2}}{A_M+B_My}&\leq1,\\
1-z&=O(X^{-2}),\\
\frac{Q^2}{2}\tan\left(\frac{2\xi}{Q^2}\right)&=\xi+O\left(\frac{\xi^3}{Q^4}\right),
\end{align*}
where the last estimate holds for $\xi/Q^2$ sufficiently small. We also recall from \cite[Prop.\ 3.1]{KelmerKontorovich:2013} that
\begin{align*}
\frac1{A_M+B_My}, \frac1{A_M+B_Myz}, \frac1{A_Mz+B_My}&=O(\norm{M}^2).
\end{align*}
It follows immediately from these observations that the inequality \eqref{tildeR-in-KAK2'} can be replaced by
\begin{equation}\label{upper-bound-on-x}
x\leq\frac1{A_M+B_Myz}=\frac1{A_M+B_My}+O\bigg(\frac{\norm{M}^6}{X^2}\bigg).
\end{equation}
Similarly, we find that \eqref{tildeR-in-KAK3''} can be replaced by
\begin{align}\label{lower-bound-on-x}
x&>\frac{B_M\sqrt{1-y^2}}{(A_Mz+B_My)\big(\frac{Q^2}{2}\tan\big(\frac{2\xi}{Q^2}\big)\big)}\nonumber\\
&=\frac{B_M\sqrt{1-y^2}}{\xi(A_Mz+B_My)}+O\bigg(\frac{\xi\norm{M}^4}{Q^4}\bigg)\\
&=\frac{B_M\sqrt{1-y^2}}{\xi(A_M+B_My)}+O\bigg(\frac{\xi\norm{M}^4}{Q^4}+\frac{\norm{M}^6}{\xi X^2}\bigg).\nonumber
\end{align}

To finish the proof of the proposition, we investigate how the error terms in \eqref{upper-bound-on-x} and \eqref{lower-bound-on-x} affect the error term in \eqref{the-volume-integral}: By an elementary calculation, we find that
\begin{align*}
&\frac{\omega_{n-1}Q^{2(n-1)}}{2^{n-1}}\int_{-1}^1(1-y^2)^{(n-3)/2}\int_{\widetilde J_{Q,\xi}(y)}x^{n-2}\,dxdy\\
&=\frac{\omega_{n-1}Q^{2(n-1)}}{2^{n-1}}\int_{-1}^1(1-y^2)^{(n-3)/2}\int_{J_{\xi}(y)}x^{n-2}\,dxdy\\
&+O\bigg(X^{n-1}+\big(\xi^{-(n-3)}+\xi^{n-1}\big)\norm{M}^{2n}Q^{2(n-3)}+\big(1+\xi^{-(n-1)}\big)\frac{\norm{M}^{2(n+1)}Q^{2(n-1)}}{X^2}\bigg).
\end{align*}
Combing this estimate with \eqref{the-volume-integral} and \eqref{truncated volume}, we obtain
\begin{align*}
&\text{vol}\left(\mathcal R_M\left(Q,\xi\right)\right)=\frac{\omega_{n-1}Q^{2(n-1)}}{2^{n-1}}\int_{-1}^1(1-y^2)^{(n-3)/2}\int_{J_{\xi}(y)}x^{n-2}\,dxdy+\mathcal E(Q)\\
&+O\bigg(X^{n-1}+\big(\xi^{-(n-3)}+\xi^{n-1}\big)\norm{M}^{2n}Q^{2(n-3)}+\big(1+\xi^{-(n-1)}\big)\frac{\norm{M}^{2(n+1)}Q^{2(n-1)}}{X^2}\bigg).
\end{align*}
In order to balance the error terms above, we choose\footnote{Note that, by \eqref{M-remark2} and taking $Q$ large enough, this is an admissible choice of $X$ (i.e.\ $X$ satisfies \eqref{Xinterval}).}
\begin{equation}\label{X}
X=\norm{M}^{2}Q^{2\frac{n-1}{n+1}},
\end{equation} 
and arrive at
\begin{align}\label{almost-done}
&\text{vol}\left(\mathcal R_M\left(Q,\xi\right)\right)=\frac{\omega_{n-1}Q^{2(n-1)}}{2^{n-1}}\int_{-1}^1(1-y^2)^{(n-3)/2}\int_{J_{\xi}(y)}x^{n-2}\,dxdy\\
&+O\bigg(\big(\xi^{-(n-3)}+\xi^{n-1}\big)\norm{M}^{2n}Q^{2(n-3)}+\big(1+\xi^{-(n-1)}\big)\norm{M}^{2(n-1)}Q^{2\frac{(n-1)^2}{n+1}}\bigg).\nonumber
\end{align}
Finally, the desired result follows from the simple observation that, for $M$ satisfying \eqref{M-remark2}, the first error term is subsumed by the second error term in \eqref{almost-done}.
\end{proof}

In order to further investigate the main term in \eqref{firstvolume}, we introduce the sets
\begin{align*}
I_1(\xi,M)&:=\left\{y\in[-1,1] : \frac{B_M\sqrt{1-y^2}}{\xi(A_M+B_My)}<1\leq\frac1{A_M+B_My}\right\},\\
I_2(\xi,M)&:=\left\{y\in[-1,1] : \frac{B_M\sqrt{1-y^2}}{\xi(A_M+B_My)}<\frac1{A_M+B_My}\leq1\right\}.
\end{align*}
$I_1(\xi,M)$ and $I_2(\xi,M)$ are unions of intervals and we recall from \cite[Lemma 3.18]{KelmerKontorovich:2013} the following more explicit description of these sets:\footnote{Recall the definitions of $\alpha_M$ and $\lambda_{\pm,M}$ from \eqref{alpha}, \eqref{lambda} and \eqref{subindex_M}.}
\begin{align*}
 I_1(\xi,M)&:=
 \begin{cases}
[-1,\lambda_{-,M}(\xi))& \textrm {if $\xi\leq C_M$},\\
[-1,\lambda_{-,M}(\xi))\cup(\lambda_{+,M}(\xi),\frac{1-A_M}{B_M}]& \textrm {if $C_M<\xi\leq B_M$},\\
[-1,\frac{1-A_M}{B_M}]& \textrm {if $B_M< \xi$},
\end{cases}\\
 I_2(\xi,M)&:=
\begin{cases}
(\alpha_M(\xi), 1]& \textrm {if $\xi\leq C_M$},\\
[\frac{1-A_M}{B_M}, -\alpha_M(\xi))\cup(\alpha_M(\xi),1]& \textrm {if $C_M<\xi\leq B_M$},\\
[\frac{1-A_M}{B_M},1]& \textrm {if $B_M< \xi$}.
\end{cases}
\end{align*}
Note in particular that 
\begin{equation}\label{I-relation}
I(\xi,t(M))=I_1(\xi,M)\cup I_2(\xi,M) 
\end{equation}
(see \eqref{set-I}). We are now ready to finish the proof of Theorem \ref{volume-theorem}.

\begin{proof}[Proof of Theorem \ref{volume-theorem}]
We call the integral in the main term of Proposition \ref{volume-proposition} $F_M(\xi)$. Given the information about the intervals $I_1(\xi,M)$ and $I_2(\xi,M)$ above, it is immediate to note that 
\begin{align}\label{explicit-F}
F_M&(\xi)=\frac1{n-1}\int_{I_1(\xi,M)}(1-y^2)^{(n-3)/2}\bigg(1-\bigg(\frac{B_M\sqrt{1-y^2}}{\xi(A_M+B_My)}\bigg)^{n-1}\bigg)\,dy\\
+&\frac1{n-1}\int_{I_2(\xi,M)}(1-y^2)^{(n-3)/2}\bigg(\bigg(\frac1{A_M+B_My}\bigg)^{n-1}-\bigg(\frac{B_M\sqrt{1-y^2}}{\xi(A_M+B_My)}\bigg)^{n-1}\bigg)\,dy.\nonumber
\end{align} 
Furthermore, we note that $F_M(\xi)\to0$ as $\xi\to0$. Indeed, using \eqref{alpha} and \eqref{lambda} in the explicit formulas for $I_1(\xi,M)$ and $I_2(\xi,M)$, we find that both intervals are of length $O(\xi^2)$ as $\xi\to0$, from which the claim follows.

Next, we compute the derivative $F_M'(\xi)$ in each of the three regimes of $I_1(\xi,M)$ and $I_2(\xi,M)$. We first consider the case $B_M<\xi$, where the computations are easier due to the fact that none of the endpoints of $I_1(\xi,M)$ and $I_2(\xi,M)$ depend on $\xi$.  Using the explicit integral description \eqref{explicit-F} of $F_M(\xi)$,
we immediately find that
\begin{align}\label{FM-1}
F_M'(\xi)=\frac1{\xi^n}\int_{I_1(\xi,M)\cup I_2(\xi,M)}\left(1-y^2\right)^{n-2}\big(y+\tfrac{A_M}{B_M}\big)^{-(n-1)}\,dy.
\end{align}

We continue by considering the case $\xi<C_M$. Again, it follows from \eqref{explicit-F} that
\begin{align*}
&F_M'(\xi)=\left(1-\lambda_{-,M}(\xi)^2\right)^{(n-3)/2}\bigg(1-\bigg(\frac{B_M\sqrt{1-\lambda_{-,M}(\xi)^2}}{\xi(A_M+B_M\lambda_{-,M}(\xi))}\bigg)^{n-1}\bigg)\frac{\lambda_{-,M}'(\xi)}{n-1}\\
&-\left(1-\alpha_M(\xi)^2\right)^{(n-3)/2}\bigg(\bigg(\frac1{A_M+B_M\alpha_M(\xi)}\bigg)^{n-1}-\bigg(\frac{B_M\sqrt{1-\alpha_M(\xi)^2}}{\xi(A_M+B_M\alpha_M(\xi))}\bigg)^{n-1}\bigg)\frac{\alpha_M'(\xi)}{n-1}\\
&+\frac1{\xi^n}\int_{I_1(\xi,M)\cup I_2(\xi,M)}\left(1-y^2\right)^{n-2}\big(y+\tfrac{A_M}{B_M}\big)^{-(n-1)}\,dy.
\end{align*}
It is now straightforward to verify, using \eqref{lambda} and \eqref{alpha} respectively, that
\begin{align*}
\frac{B_M\sqrt{1-\lambda_{\pm,M}(\xi)^2}}{\xi(A_M+B_M\lambda_{\pm,M}(\xi))}&=1,\\
\frac{B_M\sqrt{1-\alpha_M(\xi)^2}}{\xi(A_M+B_M\alpha_M(\xi))}&=\frac1{A_M+B_M\alpha_M(\xi)}.
\end{align*}
Hence, since the first two terms in the expression for $F_M'(\xi)$ vanish, we arrive at
\begin{align}\label{FM-2}
F_M'(\xi)=\frac1{\xi^n}\int_{I_1(\xi,M)\cup I_2(\xi,M)}\left(1-y^2\right)^{n-2}\big(y+\tfrac{A_M}{B_M}\big)^{-(n-1)}\,dy.
\end{align}
Furthermore, by essentially the same argument, we find that also in the case $C_M<\xi<B_M$, we have
\begin{align}\label{FM-3}
F_M'(\xi)=\frac1{\xi^n}\int_{I_1(\xi,M)\cup I_2(\xi,M)}\left(1-y^2\right)^{n-2}\big(y+\tfrac{A_M}{B_M}\big)^{-(n-1)}\,dy.
\end{align}

Finally, it follows from \eqref{I-relation} that the right-hand sides of \eqref{FM-1}, \eqref{FM-2} and \eqref{FM-3} all give the desired expression for $F_M'(\xi)$ (i.e.\ $F_M'(\xi)=f_{\xi}(t(M))$, where $f_{\xi}$ is the function defined in \eqref{DEFOFF}). Therefore, since $F_M(\xi)$ and $f_\xi(t(M))$ are continuous functions of $\xi$, and $F_M(\xi)\to0$ as $\xi\to0$, the theorem follows from the fundamental theorem of calculus. 
\end{proof}

We end this section by pointing out that the result corresponding to Theorem \ref{volume-theorem} needed in the proof of Theorem \ref{restricted-main-theorem} can be established by essentially the same proof as Theorem \ref{volume-theorem}. Recall that $\mathcal S\subset\mathcal U$ is a spherical cap (recall also that $\mathcal U$ denotes the hyperbolic unit sphere centered at $e_{n+1}$) and that the hyperbolic cone specified by the vertex $e_{n+1}$ and the cross-section $\mathcal S$ is denoted by $\mathcal C$. We are interested in determining the volume, asymptotically as $Q\to\infty$, of the set
\begin{equation}\label{RMC}
\mathcal R_{M,\mathcal C}(Q,\xi):=\left\{g\in B_Q : \norm{gM}\leq Q, ge_{n+1},gMe_{n+1}\in \mathcal C,\v{g}{gM}<\frac{2\xi}{Q^2}\right\}
\end{equation}
for $M\in\G$, $M\notin K$. We denote the volume measure on $S^{n-1}$ by $\mu_{S^{n-1}}$ and let $\phi:S^{n-1}\to\mathcal U\subset\H^n$ denote an embedding of the Euclidean sphere $S^{n-1}$ into $H^n$ preserving all angles based at the center of the sphere.

\begin{theorem}\label{volume-theorem-2}
Let $\mathcal S\subset\mathcal U$ be a spherical cap and denote the hyperbolic cone specified by the vertex $e_{n+1}$ and the cross-section $\mathcal S$ by $\mathcal C$. Then, for every $M\in\G$, $M\notin K$, and for $\xi/Q^2$ sufficiently small, we have
\begin{multline}\label{VOLUMETHM-II}
\textup{vol}\left(\mathcal R_{M,\mathcal C}\left(Q,\xi\right)\right)=\frac{\omega_{n-1}\mu_{S^{n-1}}(\phi^{-1}(\mathcal S))Q^{2(n-1)}}{\omega_n2^{n-1}}\int_0^{\xi}f_{\zeta}\left(t(M)\right)\,d\zeta\\
+O\left(\xi Q^{2(n-2)}+g(\xi)\|M\|^{2(n-1)}Q^{2\frac{(n-1)^2}{n+1}}\right)
\end{multline}
as $Q\to\infty$, where the functions $f_{\zeta}$ and $g$ are defined by \eqref{DEFOFF} and \eqref{gdef} respectively.
\end{theorem}

\begin{proof}
To prove \eqref{VOLUMETHM-II}, we first replace $ \mathcal R_{M,\mathcal C}(Q,\xi)$ by the more tractable set
\begin{equation*}
\mathcal S_{M,\mathcal C}(Q,\xi):=\left\{g\in B_Q : \norm{gM}\leq Q, ge_{n+1}\in\mathcal C,\v{g}{gM}<\frac{2\xi}{Q^2}\right\},
\end{equation*}
and notice that 
\begin{equation*}
\textup{vol}\left(\mathcal S_{M,\mathcal C}\left(Q,\xi\right)\right)-\textup{vol}\left(\mathcal R_{M,\mathcal C}\left(Q,\xi\right)\right)=O\big(\xi Q^{2(n-2)}\big).
\end{equation*} 
Finally, we determine an asymptotic formula for $\textup{vol}\left(\mathcal S_{M,\mathcal C}\left(Q,\xi\right)\right)$ in the same way as we found the formula for $\textup{vol}\left(\mathcal R_M\left(Q,\xi\right)\right)$ in the proof of Theorem \ref{volume-theorem}.
\end{proof}

\section{Approximating counts by volumes}\label{counts-to-vols}

We are now in a position where we can relate the counting of terms in the
sum
\eqref{basic-counting} to the volumes $\vol{\mathcal R_M(Q,\xi)}$ which
have been calculated in Theorem \ref{volume-theorem} . Let $X=X(M,Q)$ be a truncation parameter
satisfying
\begin{equation}\label{truncation-parameter}
  1<X<\frac{Q}{20\norm{M}}.
\end{equation}
We define
\begin{equation*}
  \mathcal R_M(Q,\xi,X):=\{g\in\mathcal R_M(Q,\xi): \norm{g}>Q/X\},
\end{equation*}
and observe that 
\begin{equation}\label{size-complement}\vol{\mathcal R_M(Q,\xi)\setminus\mathcal
  R_M(Q,\xi,X)}=O\left(\frac{Q^{2(n-1)}}{X^{2(n-1)}}\right),
\end{equation} 
since this complement is contained in $\{g\in G:\norm{g}\leq Q/X\}$.

\subsection{Fattening and slimming} Recall
the definitions of  $B_{\delta_1}$ and $D_\delta$ from \eqref{Bdelta} and \eqref{Ddelta} and note that these sets are
invariant under inversion. We consider the fattening and
slimming of $R_M(Q,\xi,X)$ by $B_{\delta_1}\times D_\delta$. More
generally: For any sets
$S,C_1,C_2\subset G$, we define the $C_1\times C_2$-fattening
$S^+$ of $S$, and the $C_1\times C_2$-slimming $S^-$ of $S$, by
\begin{align*}
 S^+:&=\bigcup_{(h_1,h_2)\in C_1\times C_2}h_1\cdot
 S\cdot h_2\\
 S^-:&=\bigcap_{(h_1,h_2)\in C_1\times
   C_2} h_1^{-1}\cdot S\cdot h_2^{-1}.
\end{align*}
It is easy to see that \begin{equation}\label{smallobs1}S^-\subseteq
  (S^-)^+\subseteq S\subseteq
(S^+)^- \subseteq S^+,\end{equation} and that if $A\subseteq B$, then 
\begin{equation}\label{smallobs2}
 A^+\subseteq B^+,\quad A^-\subseteq B^-.
\end{equation}
The next two lemmas verify that, for small values of the parameter $0<\delta=\delta(M,Q,X)$, the $B_{\delta_1}\times D_\delta$-fattening (respectively $B_{\delta_1}\times D_\delta$-slimming) of  $R_M(Q, \xi, X)$ doesn't grow (or shrink) too drastically. 

\begin{lemma} \label{fattening-slimming-control} 
Let $\xi_0>0$. 
For $\delta$ and $\delta\norm{M}^2$ sufficiently
small, there exists a constant $c>0$, such that for any $\xi\geq
\xi_0$ we have 
\begin{align*}
    R_M^+(Q,\xi,X)\subseteq R_M(\delta_2Q,\delta_3\xi,2X),\\ 
    R_M(Q,\xi,X)\subseteq R_M^-(\delta_2Q, \delta_3\xi, 2X),
\end{align*}
where
\begin{equation*}
   \quad \delta_2:=1+c\delta\norm{M}^2,\quad
    \delta_3:=1+7c\delta X^2\norm{M}^4.
\end{equation*}
\end{lemma}

\begin{proof}
 We start by noticing that, using \eqref{smallobs1} and
 \eqref{smallobs2}, any of the two inclusions implies the other. To
 prove the first inclusion, we let  $g_1=h_1gh_2$ with $g\in R_M(Q,\xi,X)$ and $(h_1,h_2)\in
 B_{\delta_1}\times D_\delta$. Note that
 since $\norm{g}>Q/X$, we have $\norm{g}>20\norm{M}$ by
 assumption \eqref{truncation-parameter}, so we are free to apply Lemma \ref{first-bounds} and Lemma
 \ref{stability-left-right}.
 
First, we observe that, by Lemma \ref{first-bounds}\eqref{bound-one} and Remark \ref{inversions-too} (recall also that $D_\delta\subseteq B_{\delta_1}$), there exist an
absolut constant $c_1>0$ such that 
\begin{equation}
  \label{first-requirement}
  \norm{g_1}^2< \norm{g}^2(1+c_1\norm{M}^2\delta)^2\leq Q^2(1+c_1\norm{M}^2\delta)^2.
\end{equation}
We observe in a similar way, using 
Lemma \ref{stability-left-right}\eqref{uno}, that
\begin{equation}
  \label{second-requirement}
  \norm{g_1M}^2<
  Q^2(1+c_2\norm{M}^2\delta)^2
\end{equation}
for another absolute constant $c_2>0$.

Next, we use Lemma \ref{stability-left-right}\eqref{dos}, basic properties of $\arctan$, and $\norm{g}>Q/X$ to see that 
\begin{align}\label{vinequality}
\v{g_1}{g_1M}&< \v{g}{gM}+c_3\delta\norm{M}^4/\norm{g}^2<
Q^{-2}(2\xi+c_3\delta X^2\norm{M}^4).
\end{align}
Letting $c=\max(c_1,c_2,(2\xi_0)^{-1}c_3)$, we observe that for
$\delta\norm{M}^2$ sufficiently small, we have  
$\delta_2^2\leq 1+3c\delta\norm{M}^2$ and
\begin{align*}
(1+c\delta X^2\norm{M}^4)\delta_2^2\leq 1+7c\delta
  X^2\norm{M}^4,
\end{align*}
which, together with \eqref{vinequality}, implies
\begin{equation}
  \label{third-requirement}
  \v{g_1}{g_1M}<\frac{2\xi\delta_3}{(Q\delta_2)^2}.
\end{equation}
We now observe that the inequalities
\eqref{first-requirement}, \eqref{second-requirement} and \eqref{third-requirement}
show that $g_1\in R_M(\delta_2Q,\delta_3\xi)$.

To prove that $g_1\in R_M(\delta_2Q,\delta_3\xi,2X)$, we note that by Lemma \ref{first-bounds}\eqref{bound-one}, Remark \ref{inversions-too} and the above choice of $c$, we have 
 \begin{equation*}
   \norm{g_1}^2>
   \frac{(1-c\delta)^2Q^2}{X^2}>
   \frac{(\delta_2Q)^2}{(2X)^2},
 \end{equation*}
where the last inequality holds for $\delta$ and $\delta\norm{M}^2$ sufficiently
small. This finishes the proof.
\end{proof}

  \begin{remark}
    Concerning the numbers $\delta_2, \delta_3$ in Lemma
    \ref{fattening-slimming-control}: we have $1\leq\delta_2=O(1)$, whereas
    a priori we only have
    $1\leq\delta_3=O(X^2\norm{M}^2)$. 
  \end{remark}
  
\begin{lemma}\label{volume-comparison} 
Let $\epsilon>0$ and fix $\xi> 0$. For $\norm{M}\geq m_0>1$, 
$\delta\norm{M}^2$ sufficiently small, and $\delta_3=1+7c\delta X^2\norm{M}^4$ bounded, we have 
\begin{align*}
&\vol{R_M(Q\delta_2^{\pm 1},\xi\delta_3^{\pm 1})}
=\vol{R_M(Q,\xi)}\\
&\quad +O_{\xi,\epsilon}\Big(Q^{2(n-1)}\delta X^2\norm{M}^{4-n+\epsilon}+\norm{M}^{2(n-1)}Q^{2\frac{(n-1)^2}{n+1}}\Big).
\end{align*}
\end{lemma}

\begin{proof}

Clearly $\xi/Q^2$, $\xi\delta_3/(Q\delta_2)^2$ all
become sufficiently small for $Q$ sufficiently large so we may apply
Theorem~\ref{volume-theorem} to see that
up to an error of 
\begin{equation*}
O\Big((g(\xi)+g(\delta_3\xi))\norm{M}^{2(n-1)}Q^{2\frac{(n-1)^2}{n+1}}\Big),
\end{equation*}
the difference 
$\vol{R_M(Q\delta_2,\xi\delta_3)} -\vol{R_M(Q,\xi) }$ is
\begin{equation}\label{beforebound}
O\left(Q^{2(n-1)}\left((\delta_2^{2(n-1)}-1)\int_0^\xi{f_\zeta(t(M))}d\zeta+\delta_2^{2(n-1)}\int_\xi^{
      \delta_3\xi}f_\zeta(t(M))d\zeta\right)\right).
\end{equation}
Bounding the integrals using Lemma~\ref{stuff}\eqref{morestuff},
we see that \eqref{beforebound} is
\begin{align*}
&O_{\epsilon}\left(Q^{2(n-1)}\left((\delta_2-1)\xi\norm{M}^{-n+\epsilon}+\xi(\delta_3-1)\norm{M}^{-n+\epsilon}\right)\right)\\
&\qquad=O_{\epsilon}\left(\xi Q^{2(n-1)}\left(\delta\norm{M}^{2-n+\epsilon}+\delta
X^2\norm{M}^{4-n+\epsilon}\right)\right)\\
&\qquad=O_{\epsilon}\left(\xi Q^{2(n-1)}\delta X^2\norm{M}^{4-n+\epsilon}\right).
\end{align*}
Substituting $Q/\delta_2$ for $Q$ and $\xi/\delta_3$ for $\xi$ in the
above (notice that this is allowed since 
$(\xi/\delta_3)/(Q/\delta_2)^2$ becomes small when $Q$ grows sufficiently large), we find
that, up to an error of 
\begin{equation*}
O\Big((g(\xi/\delta_3)+g(\xi))\norm{M}^{2(n-1)}Q^{2\frac{(n-1)^2}{n+1}}\Big),
\end{equation*}
the difference 
\begin{equation*}\vol{R_M(Q,\xi)} -\vol{R_M(Q/\delta_2,\xi/\delta_3) }\end{equation*} is also 
bounded by $O_{\epsilon}(\xi Q^{2(n-1)}\delta X^2\norm{M}^{4-n+\epsilon})$.
The result now follows easily using that
\begin{equation*}
  g(\xi)+g(\delta_3^{\pm 1}\xi)=O_\xi(\delta_3^{n-1}),
\end{equation*}
which by assumption equals $O_{\xi}(1)$.
\end{proof}

\subsection{Test functions}\label{testfunctionsection}

In this short section, we introduce two functions $\Psi_1$ and $\Psi_2$ on $\G\backslash G$ that will be of fundamental importance when we relate counts to volumes in Section \ref{countstovolumes}. However, we begin by determining the asymptotic order of decay of the volumes of the sets $B_{\delta_1}$ and $D_{\delta}$ as $\delta\to0$. Using Proposition \ref{KAKmeasure}, we immediately find that 
\begin{equation}\label{Bdeltavolume}
\vol{B_{\delta_1}}=\omega_n\int_0^{\delta}(\sinh t )^{n-1}\,dt\asymp\delta^n 
\end{equation}
for all sufficiently small $\delta$. Moreover, with a little more effort, we can also establish the following lemma:

\begin{lemma}\label{Dvolumelemma}
We have $\vol{D_{\delta}}\asymp\delta^{n^2}$ for all sufficiently small $\delta$.
\end{lemma}

\begin{proof}
Using Proposition \ref{KAKmeasure}, we obtain
\begin{align}\label{Ddeltavolume}
\vol{D_{\delta}}=\omega_n\left(\int_{K_{\delta}}dk\right)^2\int_0^{\delta}(\sinh t)^{n-1}\,dt\asymp\vol{K_{\delta}}^2\delta^n
\end{align}
for all small enough $\delta$. It remains to determine the asymptotic order of decay of $\vol{K_{\delta}}$. Recalling the definition of $K_{\delta}$ in \eqref{Kdelta}, we find that there exists a constant $C>1$ such that, for all sufficiently small $\delta$, the preimage of $K_{\delta}$ (in the Lie algebra $\text{Lie}(K)$ of $K$) under the exponential map satisfies
\begin{equation*}
\mathcal B_{\delta/C}\subset\exp^{-1}(K_{\delta})\subset\mathcal B_{C\delta},
\end{equation*} 
where $\mathcal B_{\epsilon}$ denotes the Euclidean ball of radius
$\epsilon$ centered at the origin in $\text{Lie}(K)$. Using this fact,
together with \cite[Thm.\ 1.14]{Helgason:2000a} and possibly shrinking
the size of the admissible set of parameters $\delta$, we obtain 
\begin{equation}\label{Kdeltavolume}
\vol{K_{\delta}}\asymp\delta^{n(n-1)/2},
\end{equation}
where we have used that $\dim(\text{Lie}(K))=n(n-1)/2$. Finally, combining \eqref{Ddeltavolume} and \eqref{Kdeltavolume}, we arrive at the desired result.
\end{proof}

We now let $\delta>0$ be small enough to guarantee that the asymptotics in \eqref{Bdeltavolume} and Lemma \ref{Dvolumelemma} are valid. We introduce a smooth and non-negative test function $\psi_1$ satisfying $\psi_1(k_1gk_2)=\psi_1(g)$ (i.e.\ $\psi_1$ is spherically symmetric) and
\begin{equation}\label{phi1}
\hbox{supp}\, \psi_1\subset B_{\delta_1}, \qquad \int_G\psi_1(g)\,dg=1.
\end{equation}
Furthermore, we introduce a smooth and non-negative test function $\psi_2$ satisfying
\begin{equation}\label{phi2}
\hbox{supp}\, \psi_2\subset D_{\delta}, \qquad \int_G\psi_2(g)\,dg=1. 
\end{equation}
We can, as usual, use the test functions $\psi_i$ ($i=1,2$) to construct $\G$-automorphic functions
\begin{equation}\label{automorphictestfunctions}
\Psi_i(g):=\sum_{\g\in \G}\psi_i(\g g)
\end{equation} 
in $L^2(\G\backslash G)$ satisfying $\Psi_i(\g g)=\Psi_i(g)$ for all $\g\in \G$. It is well-known that we can choose the test functions $\psi_1$ and $\psi_2$ in such a way that $\Psi_1$ and $\Psi_2$ also satisfy
\begin{equation}\label{PsiL2norm}
\|\Psi_1\|_{L^2}\asymp\vol{B_{\delta_1}}^{-1/2}\asymp\delta^{-n/2}, \quad 
\|\Psi_2\|_{L^2}\asymp\vol{D_{\delta}}^{-1/2}\asymp\delta^{-n^2/2}
\end{equation}
and
\begin{equation}\label{PsiSobolevnorm}
\mathcal{S}_l(\Psi_1)\asymp\delta^{-l}\vol{B_{\delta_1}}^{-1/2}\asymp\delta^{-l-n/2}, \quad 
\mathcal{S}_l(\Psi_2)\asymp\delta^{-l}\vol{D_{\delta}}^{-1/2}\asymp\delta^{-l-n^2/2}
\end{equation}
for any $l\in \N$ (recall \eqref{Sobolevdefinition}, \eqref{Bdeltavolume} and Lemma \ref{Dvolumelemma}). From now on we fix such an admissible pair of test functions $\psi_1$ and $\psi_2$. The asymptotics in \eqref{PsiL2norm} and \eqref{PsiSobolevnorm}, together with Theorem \ref{main-bound} and Remark \ref{Kfiniteremark}, imply the following two corollaries: 

\begin{corollary}\label{dirty-estimate-1}
Let $s_0$ be as in Theorem \ref{main-bound} and let $\Psi_1$ be defined by \eqref{automorphictestfunctions} with our fixed test function $\psi_1$. Then 
\begin{equation*}
\langle\pi(g)\Psi_1,\Psi_1\rangle_{\G\backslash G}=\frac{1}{\vol{\G\backslash G}}+O\big(\norm{g}^{2(s_0-n+1)}\delta^{-n}\big)
\end{equation*} 
for all sufficiently small $\delta>0$.
\end{corollary}

\begin{corollary}\label{dirty-estimate}
Let $s_0$ be as in Theorem \ref{main-bound} and let $\Psi_1$ and $\Psi_2$ be defined by \eqref{automorphictestfunctions} with our fixed test functions $\psi_1$ and $\psi_2$. Then there exists an integer $c_n>\frac{n(n+1)}{2}$ such that
\begin{equation*}
\langle\pi(g)\Psi_1,\Psi_2\rangle_{\G\backslash G}=\frac{1}{\vol{\G\backslash G}}+O\big(\norm{g}^{2(s_0-n+1)}\delta^{-c_n}\big)
\end{equation*}
for all sufficiently small $\delta>0$.
\end{corollary}

\subsection{Relating counts to volumes}\label{countstovolumes}

We are now ready to show that the number of elements in
$\G\cap R_M(Q,\xi)$ can be approximated by
$\vol{R_M(Q,\xi)}/\vol{\G\backslash G}$.
Recall the constant $c_n$ from  Corollary \ref{dirty-estimate}.

\begin{lemma}\label{counts-volumes}
Fix $\xi> 0$ and fix $s_0$ as in Theorem \ref{main-bound}. For $\norm{M}\geq m_0>1$, 
we have 
\begin{align*}
  \#\G\cap R_M(Q,\xi)=&\vol{R_M(Q,\xi)}/\vol{\G\backslash G}+O_\xi(Q^{a_n}\norm{M}^{b_n}),
\end{align*}
where
\begin{equation*}
  a_n={2(n-1)\left(1-\frac{n-1-s_0}{n(1+c_n)-1}\right)},\quad b_n=(n-1)\frac{4c_n}{n(1+c_n)-1}.
\end{equation*}
\end{lemma}

\begin{proof}
Let $A\subseteq G$ be a bounded set  and
consider 
\begin{equation*}F_{A}(g_1,g_2):=\sum_{\g\in\G}1_{A}(g_1^{-1}\g g_2).
\end{equation*}
We note that this function
is $\G$-invariant in both variables under multiplication from the left. We claim that 
\begin{equation}\label{second-inequality}
\# \G\cap A^-\leq
\langle F_{A}, \Psi_1\otimes \Psi_2\rangle_{\G\backslash G\times
\G\backslash G }\leq \# \G\cap A^+.
\end{equation}
Here $\Psi_i$ is defined in \eqref{automorphictestfunctions} and
$A^+$ (resp.\ $A^-$) is the $B_{\delta_1}\times D_\delta$--fattening
(resp.\ slimming) of $A$.

First, we unfold the functions $\Psi_i$ in the middle expression and find that this inner product equals
\begin{equation}\label{intermediate}
\sum_{\g\in \G} \int_{G}\int_{G}1_{A}(g_1^{-1}\g g_2)\psi_1(g_1)\psi_2(g_2)\,dg_1dg_2.
\end{equation}
To see the left inequality in \eqref{second-inequality}, we now note that for every $\g\in\G\cap A^-$ we have, since $B_{\delta_1}$ is
symmetric under inversion, that $g_1^{-1}\g g_2\in A$ whenever
$(g_1,g_2)\in B_{\delta_1}\times D_\delta$. Therefore, using
\eqref{phi1} and \eqref{phi2}, we find that every term
in the sum \eqref{intermediate} corresponding to such a $\g$ contributes by 1. 

To see the right inequality in \eqref{second-inequality}, we note that for $\g$ to give a non-zero contribution to the sum
\eqref{intermediate} the requirements
\eqref{phi1},~\eqref{phi2} imply that there must exist a pair
$(g_1,g_2)\in B_{\delta_1}\times D_\delta$ such that $g_1^{-1}\g
g_2\in A$. But this implies, since $D_{\delta}$ is symmetric under
inversion, that $\g\in A^+$. Moreover,
given $\gamma\in A^+$ the corresponding integral in \eqref{intermediate} can be at most 1, again by \eqref{phi1}, \eqref{phi2}. 

On the other hand, analyzing the inner product in \eqref{second-inequality} by unfolding the $\G$-sum defining $F_A$ and making the
change of variables $g=g_1^{-1}g_2$, we find  
\begin{align*}
\langle F_{A}, \Psi_1\otimes \Psi_2\rangle_{\G\backslash G\times \G\backslash G }=&\int_{A}\int_{\G\backslash G}\Psi_1(g_1)\Psi_2(g_1g)\,dg_1dg\\
=&\int_{A}\langle\pi(g)\Psi_2, \Psi_1\rangle_{\G\backslash G}\,dg,
\end{align*}
where $\pi$ denotes the right regular representation. Assume now that the parameters $\delta$ and $X$ are satisfying
\begin{equation}\label{all-assumptions}
\delta\norm{M}^2\textrm{ small, Eq.\ \eqref{truncation-parameter},  and } \delta_3=1+7c\delta X^2\norm{M}^4=O(1).
\end{equation}
Here the constant $c$ is as in Lemma~\ref{fattening-slimming-control} and Lemma~\ref{volume-comparison}. It follows from the above discussion and
Lemma~\ref{fattening-slimming-control} that\footnote{Notice that we are free to apply Lemma 5.1 since $\xi/\delta_3$ by assumption \eqref{all-assumptions} is bounded from below.}  
\begin{equation*}
\int_{R_M(Q/\delta_2,\xi/\delta_3,X/2)}\!\!\!\!\!\!\!\!\!\!\!\!\!\!\!\!\!\!\!\!\!\!\!\!\!\!\!\!\!\!\!\!\langle\pi(g)\Psi_2,\Psi_1\rangle_{\G\backslash G}\,dg
\leq \# \G\cap R_M(Q,\xi,X)
  \leq \int_{R_M(Q\delta_2,\xi\delta_3,2X)}\!\!\!\!\!\!\!\!\!\!\!\!\!\!\!\!\!\!\!\!\!\!\!\!\!\!\!\!\langle\pi(g)\Psi_2,\Psi_1\rangle_{\G\backslash G}\,dg.
\end{equation*}
Once this has been established, we only need to use the decay of matrix
coefficients (Corollary \ref{dirty-estimate}) to approximate the
integrals by volumes, and then the volume estimates from Lemma ~\ref{volume-comparison} to estimate
the relevant count. To be more precise:

Using Corollary \ref{dirty-estimate}, 
and $R_M(\delta_2Q,\delta_3\xi,2X)\subset B_{2Q}$ for $\delta\norm{M}^2$ sufficiently
small, we find that since 
\begin{equation*}
  \int_{R_M(Q\delta_2^{\pm 1},\xi\delta_3^{\pm 1},2^{\pm 1}X)}\norm{g}^{2(s_0-n+1)}\,dg=O(Q^{2s_0}) 
\end{equation*}
we have 
\begin{equation*}
\frac{\vol{R_M(Q/\delta_2,
    \xi/\delta_3,X/2)}}{\vol{\G\backslash G}}+O(\delta^{-c_n}Q^{2s_0}) \leq  \# \G\cap R_M(Q,\xi,X)
\end{equation*}
and
\begin{equation*}
    \# \G\cap R_M(Q,\xi,X)\leq \frac{\vol{R_M(Q\delta_2, \xi\delta_3,2X)}}{\vol{\G\backslash G}}+O(\delta^{-c_n}Q^{2s_0}).
\end{equation*}
Furthermore, using \eqref{size-complement}  and Lemma~\ref{volume-comparison} with a fixed small $\epsilon$, we see that 
\begin{align*}
&\frac{\vol{R_M(Q\delta_2^{\pm 1}, \xi\delta_3^{\pm 1},2^{\pm1}X)}}{\vol{\G\backslash G}}=\frac{\vol{R_M(Q\delta_2^{\pm 1}, \xi\delta_3^{\pm 1})}}{\vol{\G\backslash G}}+O\left(\frac{Q^{2(n-1)}}{X^{2(n-1)}}\right)\\
&=\frac{\vol{R_M(Q, \xi)}}{\vol{\G\backslash G}}+O_\xi\left(Q^{2(n-1)}\delta X^2\norm{M}^{4-n+\epsilon}+\norm{M}^{2(n-1)}Q^{2\frac{(n-1)^2}{n+1}}+\frac{Q^{2(n-1)}}{X^{2(n-1)}}\right).
\end{align*}
In order to control the error terms above, we first balance 
$Q^{2(n-1)}\delta X^2\norm{M}^{4} $ with $Q^{2(n-1)}/X^{2(n-1)}$ and
find
\begin{equation}\label{choice of delta}
  \delta=X^{-2n}\norm{M}^{-4}.
\end{equation}
We omit the extra decay in $\norm{M}^{-n+\epsilon}$ in order to be
able to  verify that $\delta_3$ is bounded; with the above choice of $\delta$ we have
$\delta_3=1+7cX^{-2(n-1)}$ and by \eqref{truncation-parameter} this is
indeed bounded. Using that 
\begin{equation*}
  \#\G\cap \mathcal R_M(Q,\xi,X)= \#\G\cap \mathcal R_M(Q,\xi)+O\left(\frac{Q^{2(n-1)}}{X^{2(n-1)}}\right),
\end{equation*}
we find, with $\delta$ as in \eqref{choice of delta}, that 
\begin{align}\label{almost-there}
  \#\G\cap \mathcal
  R_M&(Q,\xi)=\frac{\vol{R_M(Q,\xi)}}{\vol{\G\backslash G}}\\ 
&+O_\xi\left(\frac{Q^{2(n-1)}}{X^{2(n-1)}}+\delta^{-c_n}Q^{2s_0}
  +\norm{M}^{2(n-1)}Q^{2\frac{(n-1)^2}{n+1}}\right).\nonumber
\end{align}
We now balance $X$ between the first two error terms and find 
\begin{equation*}
X=Q^{\frac{n-1-s_0}{n(1+c_n)-1}}\norm{M}^{-\frac{2c_n}{n(1+c_n)-1}}.
\end{equation*}
We can certainly assume that $\norm M<Q^{(2(n-1)-a_n)/b_n}=Q^{\frac{n-1-s_0}{2c_n}}$, since
otherwise the claim of the lemma is trivial. 
With these choices of parameters, a computation, using also that
$c_n>\frac{n(n+1)}{2}$, verifies that for $Q$ sufficiently large
\eqref{all-assumptions} is indeed satisfied. It is also
straightforward to verify, again using $c_n>\frac{n(n+1)}{2}$, that 
\begin{equation*}
  \norm{M}^{2(n-1)}Q^{2\frac{(n-1)^2}{n+1}}\leq Q^{a_n}\norm{M}^{b_n}.
\end{equation*}
 Inserting these values of
$X$ and $\delta$ in \eqref{almost-there}, we arrive at the claim.
\end{proof}

\subsection{More on the relation between counts and volumes}

In this section we briefly discuss a modified version of Lemma
\ref{counts-volumes} needed in the proof of Theorem
\ref{restricted-main-theorem}. We recall that $\mathcal U$ denotes the
hyperbolic unit sphere centered at $e_{n+1}$. Let $\mathcal
S\subset\mathcal U$ be a spherical cap with opening angle $\theta<\pi$,
and denote the hyperbolic cone specified by the vertex $e_{n+1}$ and the cross-section $\mathcal S$ by $\mathcal C$. 

We are interested in counting the number of points in the intersection of $\G$ with the set $\mathcal R_{M,\mathcal C}(Q,\xi)$ defined in \eqref{RMC}. As in the case studied above we consider, for positive numbers $X$ satisfying \eqref{truncation-parameter}, the truncation
\begin{equation*}
\mathcal R_{M,\mathcal C}(Q,\xi,X):=\left\{g\in\mathcal R_{M,\mathcal C}(Q,\xi) : \norm{g}>Q/X\right\}.
\end{equation*}
We note that in contrast to $\mathcal R_{M}(Q,\xi,X)$, this set is not left $K$-invariant and hence we have to adapt the fattening and slimming described in Lemma \ref{fattening-slimming-control} slightly. We need to consider, for small parameters $\delta>0$, the $D_{\delta}\times D_{\delta}$-fattening (respectively $D_{\delta}\times D_{\delta}$-slimming) of $\mathcal R_{M,\mathcal C}(Q,\xi,X)$. It turns out that both the result and the proof of Lemma \ref{fattening-slimming-control} carries over to the present situation except for one detail. If we let $g_1=h_1gh_2$ with $g\in R_{M,\mathcal C}(Q,\xi,X)$ and $h_1,h_2\in D_\delta$, then $g_1e_{n+1}$ and $g_1Me_{n+1}$ need not be contained in the cone $\mathcal C$. In order to compensate for this fact, we have to enlarge the cone in the right-hand sides of the statements corresponding to Lemma \ref{fattening-slimming-control}. To be more precise:       
Let $\mathcal C=\{x\in\H^n : \v{x}{g'e_{n+1}}<\theta\}$ for a suitable
$g'\in G$ not fixing the base point $e_{n+1}$. Then, using Lemma
\ref{fattening-stability}\eqref{v-growth} and the triangle inequality
(Proposition~\ref{angle-props}\eqref{trig-ineq}), it is possible to show that 
\begin{align*}
R_{M,\mathcal C}^+(Q,\xi,X)\subseteq R_{M,\mathcal C_1}(\delta_2Q,\delta_3\xi,2X),\\ 
R_{M,\mathcal C}(Q,\xi,X)\subseteq R_{M,\mathcal C_1}^-(\delta_2Q, \delta_3\xi, 2X),
\end{align*}
where $\mathcal C_1=\{x\in\H^n : \v{x}{g'e_{n+1}}<\theta+\delta_4\}$ and $\delta_4=\kappa\delta+\frac{2\xi\delta_3}{(Q\delta_2)^2}$ for an absolute constant $\kappa$ (here $\delta_2$ and $\delta_3$ are as in Lemma \ref{fattening-slimming-control}). Recall that we may assume that $\delta_3$ is bounded and that $\xi/Q^2$ is sufficiently small; hence $\delta_4$ is small. The fact that we have to consider $\mathcal C_1$ on the right-hand sides above introduces an extra approximation step when we generalize Lemma \ref{volume-comparison}; we first compare (for example) $\vol{R_{M,\mathcal C_1}(\delta_2Q,\delta_3\xi)}$ to $\vol{R_{M,\mathcal C}(\delta_2Q,\delta_3\xi)}$ using an elementary estimate and then compare $\vol{R_{M,\mathcal C}(\delta_2Q,\delta_3\xi)}$ to $\vol{R_{M,\mathcal C}(Q,\xi)}$ using Theorem \ref{volume-theorem-2}. The details are straightforward.

Turning to the generalization of Lemma \ref{counts-volumes}, we need to replace the test function $\Psi_1\otimes\Psi_2$ with $\Psi_2\otimes\Psi_2$. It follows that we need to consider the matrix coefficient $\langle\pi(g)\Psi_2,\Psi_2\rangle_{\G\backslash G}$. Using Theorem \ref{main-bound} and \eqref{PsiSobolevnorm}, we find that there exists an integer $d_n>n^2+1$ such that
\begin{equation}\label{Psi2Psi2matrixcoefficient}
\langle\pi(g)\Psi_2,\Psi_2\rangle_{\G\backslash G}=\frac{1}{\vol{\G\backslash G}}+O\big(\norm{g}^{2(s_0-n+1)}\delta^{-d_n}\big)
\end{equation}
for all sufficiently small $\delta>0$. Noticing that the rest of the proof can be generalized with only minor changes, we arrive at the following result: 

\begin{lemma}\label{counts-volumes-2}
Fix $\xi> 0$ and fix $s_0$ as in Theorem \ref{main-bound}. Let $d_n$ be as in \eqref{Psi2Psi2matrixcoefficient}. Let $\mathcal S\subset\mathcal U$ be a spherical cap and denote the hyperbolic cone specified by the vertex $e_{n+1}$ and the cross-section $\mathcal S$ by $\mathcal C$. Then, for $\norm{M}\geq m_0>1$, 
we have 
\begin{align*}
  \#\G\cap R_{M,\mathcal C}(Q,\xi)=&\vol{R_{M,\mathcal C}(Q,\xi)}/\vol{\G\backslash G}+O_{\xi}\big(Q^{a_n'}\norm{M}^{b_n'}\big),
\end{align*}
where
\begin{equation*}
a_n'={2(n-1)\left(1-\frac{n-1-s_0}{n(1+d_n)-1}\right)},\quad b_n'=(n-1)\frac{4d_n}{n(1+d_n)-1}.
\end{equation*}
\end{lemma}

\section{Proofs of the main theorems}\label{proofs}

In this section we finish the proofs of the main results stated in the introduction.

\subsection{Proof of Theorem \ref{main theorem}}

Recall from \eqref{pair-correlation-function-intro} and \eqref{basic-counting} that the basic object we need to investigate is the sum
\begin{equation}\label{NQ}
\#N_{2,Q}(\xi)=\sum_{\substack{M\in \G\\ M\notin K}}\#\G\cap \mathcal R_M(Q,k_{n,\Gamma}\xi).
\end{equation} 
Recall also the constant $k_{n,\Gamma}$ (see \eqref{kn-constant}) and the function 
\begin{equation*}
f_{\xi}(t)=\frac1{\xi^n}\int_{I(\xi,t)}\left(1-y^2\right)^{n-2}\left(y+\coth t\right)^{-(n-1)}\,dy
\end{equation*}
defined in \eqref{DEFOFF} (see also \eqref{set-I}). We will prove the
following precise version of Theorem \ref{main theorem} establishing
the existence and properties of the limit
\begin{equation}
  \label{R2limit}
R_2(\xi):=\lim_{Q\to\infty}R_{2,Q}(\xi).  
\end{equation}

\begin{theorem}\label{main theorem precise version}
Let $n\geq 2$ and let $\G\subset G$ be a lattice. The limit \eqref{R2limit}, defining the pair correlation function $R_2$, exists and is differentiable. In fact, the pair correlation density function $g_2$ is given by
\begin{align}\label{g2formula}
g_2(\xi)=\frac{(n-1)\omega_{n-1}k_{n,\Gamma}}{\omega_n}\sum_{M\in\Gamma}f_{\xi
  k_{n,\Gamma}}(t(M)).
\end{align}
Furthermore, there exists a real number $\nu>0$, depending only on $n$ and the spectral gap for the group $\G$, satisfying, for fixed $\xi>0$ and $Q\to\infty$, the relation
\begin{equation}\label{NQdefinition}
\#N_{2,Q}(\xi )=\frac{\omega_nQ^{2(n-1)}}{2^{n-1}(n-1)\vol{\G\backslash G}}R_2(\xi)+O_{\xi}\big(Q^{2(n-1)-\nu}\big).
\end{equation} 
\end{theorem}

\begin{remark}
 It follows immediately from \eqref{g2formula}  that the function $F_{\xi}$ in Theorem \ref{main theorem} is given by 
\begin{equation}\label{Fxi}
F_{\xi}(t):=\frac{(n-1)\omega_{n-1}k_{n,\Gamma}}{\omega_n}f_{\xi k_{n,\Gamma}}(t).
\end{equation}
\end{remark}

\begin{remark}
The proof of Theorem \ref{main theorem precise version} shows that \eqref{NQdefinition} holds with any exponent $\nu$ satisfying 
\begin{equation*}
\nu<\frac{(n-1)(n-1-s_0)}{n(1+c_n)+c_n-1}.
\end{equation*} 
\end{remark}

We begin by proving an elementary lemma.

\begin{lemma}\label{dyadic lemma}
For each $g'\in G$, we have
\begin{equation*}
  \#\left\{\g\in N_\G(Q) : \v{\g}{g'}<\frac{2\xi}{Q^2}\right\}=O_{\xi}(\log Q).
\end{equation*}
\end{lemma}

\begin{proof}
We let $3<Q_1<Q$ and consider the quantity
\begin{equation*}
\mathcal L:=\#\left\{\g\in \G : Q_1<\norm{\g}\leq 2Q_1, \v{\g}{g'}<\frac{2\xi}{Q^2}\right\}.
\end{equation*}
By choosing a fixed $\delta>0$, depending only on $\G$ and small enough to ensure that $\gamma_1
B_{\delta_1}\cap\gamma_2 B_{\delta_1}=\emptyset$ for all $\gamma_1\neq\gamma_2$ in $\G$ (recall the definition of the ball $B_{\delta_1}$ in \eqref{Bdelta}), we write
\begin{equation*}
\mathcal L=\frac{1}{\vol{B_{\delta_1}}}\sum_{\substack{\g\in\G\\ Q_1<\norm{\g}\leq 2Q_1\\\v{\g}{g'}<2\xi/Q^2}}\vol{\g B_{\delta_1}}.
\end{equation*}
Using Lemma \ref{first-bounds} and Proposition \ref{angle-props},  possibly decreasing the value of $\delta$, we obtain
\begin{equation*}
\mathcal{L}\leq \frac{1}{\vol{B_{\delta_1}}}\textup{vol}\left(\left\{g\in G : \norm{g}<3Q_1, \v{g}{g'}<\frac{2\xi+c\delta}{Q_1^2}\right\}\right),
\end{equation*}
with an absolute constant $c>0$. Estimating the volume in the numerator, we find that
$\mathcal L=O(1+\xi^{n-1})=O_{\xi}(1)$ independently of $g'$ (recall that all our implied constants are allowed to depend on $\G$). Hence, the desired result follows from a dyadic decomposing of the condition $\norm{\g}\leq Q$ (estimating the contribution from elements $\gamma$ satisfying, say, $\norm{\gamma}\leq 6$ trivially). 
\end{proof}

We continue by giving an upper bound on the tail of the sum
\eqref{NQ}. To be more precise, we consider 
\begin{equation*}
  \mathcal{E}_{Q,T}(\xi):=\sum_{\norm{M}\geq T}\#\G\cap \mathcal{R}_M(Q,\xi), 
\end{equation*}
where $0<T=T(Q)<Q$ is a parameter tending to infinity with Q. 

\begin{lemma}\label{error term E}
Fix $\xi$ and let $T<Q$. Then, for $T$ sufficiently large, we have 
\begin{equation*}
\mathcal{E}_{Q,T}(\xi)=O_{\xi}\left(\frac{Q^{2(n-1)}\log Q}{T^{2(n-1)}}\right).
\end{equation*}
\end{lemma}

\begin{proof}
Note that $\mathcal{E}_{Q,T}(\xi)$ equals the cardinality of the set
\begin{equation*}
\mathcal S:= \left\{(\g,\g')\in N_{\G}(Q)^2 : \norm{\g^{-1}\g'}\geq T, \v{\g}{\g'}<\frac{2\xi}{Q^2}\right\}.
\end{equation*}
For $(\gamma,\gamma')\in \mathcal S$ we find, using Proposition \ref{change-under-right-mult} and the relation $\cos(\pi-\v{\g}{\g'})=-1+O(\xi^2/Q^4)$, that
\begin{align*}
T^2\leq \norm{\g^{-1}\g'}^2&=2\big(\cosh t(\g)\cosh t(\g')+\cos(\pi-\v{\g}{\g'})\sinh t(\g)\sinh t(\g')\big)\\
&=2\cosh{(t(\g)-t(\g'))}+O_{\xi}(1).
\end{align*}
Assuming $t(\g)-t(\g')\geq0$, it follows
that for $T$ sufficiently large (depending only on $\xi$), we have
$e^{t(\g)-t(\g')}\geq  T^2/2$ and therefore, since $t(\g)\leq 2\log Q$, we also have
$t(\g')\leq 2\log (Q/T)+\log2$. In particular, we readily have
$\norm{\g'}^2\leq 3\frac{Q^2}{T^2}$. Hence, we conclude
\begin{align*}
\mathcal{E}_{Q,T}(\xi)\leq&2\#\left\{(\g,\g')\in\G^2 : \norm{\g}\leq Q, \norm{\g'}\leq\sqrt 3\frac{Q}{T}, \v{\g}{\g'}<\frac{2\xi}{Q^2}\right\}\\
=&2\sum_{\substack{\g'\in\G\\\norm{\g'}\leq\sqrt 3 Q/T}}\#\left\{\g\in N_\G(Q) : \v{\g}{\g'}<\frac{2\xi}{Q^2}\right\},
\end{align*}
and the result follows immediately from Lemma \ref{dyadic lemma}.
\end{proof}

We are now ready to complete the proof of Theorem \ref{main theorem precise version}.

\begin{proof}[Proof of Theorem \ref{main theorem precise version}]
Recall from \eqref{NQdefinition} that our goal is an asymptotic expansion of $\#N_{2,Q}(\xi)$, with a power saving error term. We introduce a positive parameter $T<Q$ (at the end of the proof we will determine a value of $T$ that balance our error terms; see \eqref{value of T}), and apply Lemma \ref{error term E} to get
\begin{equation*}
\#N_{2,Q}\left(\frac{\xi}{k_{n,\Gamma}}\right)=\sum_{\substack{M\in \G\\ M\notin K, \norm{M}<T}}\#\G\cap \mathcal R_M(Q,\xi)+O_{\xi}\left(\frac{Q^{2(n-1)}\log Q}{T^{2(n-1)}}\right).
\end{equation*}
Applying also Lemma \ref{counts-volumes} and Theorem \ref{volume-theorem} yields
\begin{multline*}
\#N_{2,Q}\left(\frac{\xi}{k_{n,\Gamma}}\right)=\frac{\omega_{n-1}Q^{2(n-1)}}{2^{n-1}\vol{\G\backslash G}}
\sum_{\substack{M\in \G\\ \norm{M}<T}}\int_0^{\xi}f_{\zeta}(t(M))\,d\zeta\\
+O_{\xi}\left(T^{4(n-1)}Q^{2\frac{(n-1)^2}{n+1}}+T^{2(n-1)+b_n}Q^{a_n}+\frac{Q^{2(n-1)}\log Q}{T^{2(n-1)}}\right),
\end{multline*}
where $a_n$ and $b_n$ are as in Lemma \ref{counts-volumes}. Furthermore, using Lemma \ref{stuff}\eqref{stuff-onepointfive}, we find that we may drop the condition $\norm{M}<T$ in the above summation; the error term introduced in this step is subsumed in the error term $O_{\xi}((Q/T)^{2(n-1)}\log Q)$. Thus
\begin{multline}\label{NQestimate}
\#N_{2,Q}\left(\frac{\xi}{k_{n,\Gamma}}\right)=\frac{\omega_{n-1}Q^{2(n-1)}}{2^{n-1}\vol{\G\backslash G}}
\sum_{\substack{M\in \G}}\int_0^{\xi}f_{\zeta}(t(M))\,d\zeta\\
+O_{\xi}\left(T^{4(n-1)}Q^{2\frac{(n-1)^2}{n+1}}+T^{2(n-1)+b_n}Q^{a_n}+\frac{Q^{2(n-1)}\log Q}{T^{2(n-1)}}\right).
\end{multline}
We balance the last two error terms in \eqref{NQestimate} by choosing
\begin{equation}\label{value of T}
T=Q^{\frac{2(n-1)-a_n}{4(n-1)+b_n}},
\end{equation} 
and with this choice of $T$, using also that $c_n>\frac{n(n+1)}{2}$, we readily verify that 
\begin{equation*}
T^{4(n-1)}Q^{2\frac{(n-1)^2}{n+1}}\leq\frac{Q^{2(n-1)}}{T^{2(n-1)}}.
\end{equation*}
Hence we can drop the first error term in \eqref{NQestimate} and we arrive at
\begin{multline}\label{NQresult}\#N_{2,Q}\left(\frac{\xi}{k_{n,\Gamma}}\right)=\frac{\omega_{n-1}Q^{2(n-1)}}{2^{n-1}\vol{\G\backslash G}}
\sum_{\substack{M\in \G}}\int_0^{\xi}f_{\zeta}(t(M))\,d\zeta\\
+O_{\xi,\epsilon}\left(Q^{2(n-1)(\frac{2(n-1)+a_n+b_n}{4(n-1)+b_n})+\epsilon}\right).
\end{multline}
It is now straightforward to verify, using \eqref{pair-correlation-function-intro}, \eqref{R2limit} and Remark \ref{LPremark}, that this confirms the claim in \eqref{NQdefinition} with any exponent $\nu$ satisfying 
\begin{equation*}
\nu<\frac{(n-1)(n-1-s_0)}{n(1+c_n)+c_n-1}.
\end{equation*} 
Since also the remaining part of Theorem \ref{main theorem precise version} follows immediately from \eqref{NQresult}, the proof is complete. 
\end{proof}

\begin{remark}
Let us point out that the proof of Theorem \ref{main theorem precise version} can, with only minimal changes (e.g.\ replacing the use of Lemma \ref{counts-volumes} and Theorem \ref{volume-theorem} by applications of Lemma \ref{counts-volumes-2} and Theorem \ref{volume-theorem-2} respectively), be turned into a proof of Theorem \ref{restricted-main-theorem}. 
\end{remark}

\subsection{Proof of Theorem \ref{main theorem-asymptotics}}

We recall from \eqref{g2formula} that 
\begin{equation}\label{repeatedg_2}
g_2\Big(\frac{\xi}{k_{n,\Gamma}}\Big)=\frac{(n-1)\omega_{n-1}k_{n,\Gamma}}{\omega_n}\sum_{M\in\Gamma}f_{\xi}(t(M)).
\end{equation}
Our main task is to prove the following asymptotic formula:

\begin{lemma}\label{kernellemma}
Let $\epsilon>0$ and let $s_0$ be as in Theorem \ref{main-bound}. Then we have 
\begin{equation*}
\sum_{M\in\Gamma}f_{\xi}(t(M))=\frac{1}{\vol{\G\backslash G}}\int_Gf_{\xi}(t(g))\,dg+O_{\epsilon}\left(\xi^{n-2+\frac{2(s_0-n+1)}{n+1}+\epsilon}\right)
\end{equation*}
as $\xi\to\infty$.
\end{lemma}

\begin{proof}
Let $\delta>0$ be a small parameter. As in the proof of Lemma \ref{fattening-stability}, we find that $t(g_1^{-1}Mg_2)-t(M)=O(\delta)$ for all $g_1,g_2\in B_{\delta_1}$. Using this fact, together with Lemma \ref{perturbationlemma} and Remark \ref{LPremark}, we obtain
\begin{align}\label{kernelestimate}
\Big|\sum_{M\in\Gamma}f_{\xi}(t(&g_1^{-1}Mg_2))-\sum_{M\in\Gamma}f_{\xi}(t(M))\Big|
\leq\sum_{M\in\Gamma}\big|f_{\xi}(t(g_1^{-1}Mg_2))-f_{\xi}(t(M))\big|\nonumber\\
&\ll\sum_{\norm{M}\ll\xi^{1/2}}\delta^{1/2}\xi^{-n}+\sum_{\norm{M}\ll\xi}\delta\xi^{-n}+\sum_{\norm{M}\gg\xi}\delta\xi^{n-2}\norm{M}^{-4(n-1)}\\
&\ll\delta^{1/2}\xi^{-1}+\delta\xi^{n-2}+\delta\xi^{-n}\ll\delta^{1/2}\xi^{-1}+\delta\xi^{n-2}.\nonumber
\end{align}

For $\delta>0$ sufficiently small, we again consider the spherically symmetric test function $\psi_1$ (with support contained in $B_{\delta_1}$) and the corresponding $\G$-automorphic function $\Psi_1$ defined in Section \ref{testfunctionsection}. We recall in particular that $\int_G\psi_1(g)\,dg=1$ and that $\|\Psi_1\|\asymp\delta^{-n/2}$. Using the functions $\psi_1$ and $\Psi_1$, together with the estimate \eqref{kernelestimate}, we get
\begin{align}\label{rightlefthandside}
\sum_{M\in\Gamma}f_{\xi}(t(M))=\Big\langle\sum_{M\in\Gamma}f_{\xi}(t(g_1^{-1}Mg_2)), \Psi_1\otimes \Psi_1(g_1,g_2)\Big\rangle_{\G\backslash G\times\G\backslash G }\\
+O\big(\delta^{1/2}\xi^{-1}+\delta\xi^{n-2}\big).\nonumber
\end{align}
Unfolding the summation in $\sum_{M\in\Gamma}f_{\xi}(t(g_1^{-1}Mg_2))$, making the change of variables $g=g_1^{-1}g_2$, and applying Corollary \ref{dirty-estimate-1} and Theorem \ref{theoremonasymptotics}\eqref{upperbounds}, we find that
\begin{align}\label{righthandsideestimate}
\Big\langle\sum_{M\in\Gamma}&f_{\xi}(t(g_1^{-1}Mg_2)),\Psi_1\otimes \Psi_1(g_1,g_2)\Big\rangle_{\G\backslash G\times\G\backslash G }\\
&=\int_Gf_{\xi}(t(g))\langle\pi(g)\Psi_1,\Psi_1\rangle_{\G\backslash G}\,dg\nonumber\\
&=\frac1{\vol{\G\backslash G}}\int_Gf_{\xi}(t(g))\,dg+O\left(\delta^{-n}\int_Gf_{\xi}(t(g))\norm{g}^{2(s_0-n+1)}\,dg\right)\nonumber\\
&=\frac1{\vol{\G\backslash G}}\int_Gf_{\xi}(t(g))\,dg+O_{\epsilon}\left(\delta^{-n}\xi^{2s_0-n+\epsilon}\right),\nonumber
\end{align} 
where again $\pi$ denotes the right regular representation on $G$. We balance the error terms in \eqref{rightlefthandside} and \eqref{righthandsideestimate} by choosing
\begin{equation*}
\delta=\xi^{\frac{2(s_0-n+1)}{n+1}},
\end{equation*}
and arrive at the asymptotic formula
\begin{align*}
\sum_{M\in\Gamma}f_{\xi}(t(M))=\frac1{\vol{\G\backslash G}}\int_Gf_{\xi}(t(g))\,dg+O_{\epsilon}\left(\xi^{n-2+\frac{2(s_0-n+1)}{n+1}+\epsilon}\right),
\end{align*}
which is the desired result.
\end{proof}

We are now in position to finish the proof of Theorem \ref{main theorem-asymptotics}.

\begin{proof}[Proof of Theorem \ref{main theorem-asymptotics}]
Using \eqref{repeatedg_2}, Lemma \ref{kernellemma} and Theorem \ref{theoremonasymptotics}\eqref{asymptotics}, we find that
\begin{align*}
g_2\Big(\frac{\xi}{k_{n,\Gamma}}\Big)=\frac{\omega_{n-1}k_{n,\G}}{(n-1)\vol{\G\backslash G}}\xi^{n-2}+O_{\epsilon}\left(\xi^{-1+\epsilon}+\xi^{n-4}+\xi^{n-2+\frac{2(s_0-n+1)}{n+1}+\epsilon}\right).
\end{align*}
Hence
\begin{align}\label{g_2finalasymptotics}
g_2(\xi)&=\frac{\omega_{n-1}k_{n,\G}^{n-1}}{(n-1)\vol{\G\backslash G}}\xi^{n-2}+O_{\epsilon}\left(\xi^{-1+\epsilon}+\xi^{n-4}+\xi^{n-2+\frac{2(s_0-n+1)}{n+1}+\epsilon}\right)\\
&=(n-1)\xi^{n-2}+O_{\epsilon}\left(\xi^{n-2+\frac{2(s_0-n+1)}{n+1}+\epsilon}\right).\nonumber
\end{align}
Finally, noting that the condition specifying $s_0$ is an open condition, we find that the conclusion in \eqref{g_2finalasymptotics} holds also with $s_0$ replaced by any smaller number still admissible in the statement of Theorem \ref{main-bound}. In particular we can, by allowing the implied constant to depend on $s_0$, drop the $\epsilon$ in the above error term. This finishes the proof.
\end{proof}

\subsection{Proof of Theorem \ref{g_2at0theorem}}

 Kelmer and Kontorovich  prove in \cite[p.\ 8]{KelmerKontorovich:2013} that in the case $n=2$ the pair correlation density tends to the strictly positive value
\begin{equation*}
g_2(0)=\frac{\vol{\G\backslash G}}{\pi}\sum_{\substack{M\in\G\\ t(M)>0}}\frac{1}{e^{2t(M)}-1}
\end{equation*}
as $\xi\to0$ (see also \cite[Eq.\
(1.3)]{BocaPopaZaharescu:2013})\footnote{Note that we get the same
  value for $g_2(0)$ as Kelmer and Kontorovich even though our
  normalization of the pair correlation function is slightly different
  from the one in \cite{KelmerKontorovich:2013}.}. Theorem
\ref{g_2at0theorem} now follows directly from Theorem \ref{main
  theorem} and Lemma \ref{stuff}\eqref{stuff-onepointfive}.

\section{Geometric and spectral information contained in $g_2$}\label{friday-afternoon}

For a lattice $\G=\{M_i : i\in \N\}\cup\{I\}\subseteq G$, we call the sequence  
$$0=t(I)<t(M_1)\leq t(M_2)\leq t(M_3)\leq \ldots$$
the \emph{lattice length spectrum} of $\G$ (not to be confused with the length spectrum). By \eqref{LP-intro}, this set determines $n$ and $\vol{\G\backslash G}$ and, by Theorem \ref{main theorem}, it therefore also determines the pair correlation density function.  

 In fact, the opposite is also true. Given a pair correlation density
 function $g_2(\xi)$  and a volume $\vol{\Gamma\backslash G}$, we can
 find the lattice length spectrum as follows: We find
 $n$ from Theorem \ref{main theorem-asymptotics}. By Lemma
 \ref{ob}, we know that $g_2$ is non-differentiable precisely at the points
 $2k_n^{-1}\sinh(t(M)/2)$ and $k_n^{-1}\sinh(t(M))$ (here $M$ runs through the nontrivial elements of $\G$). The smallest value
 of $\xi$ for which $g_2(\xi)$ is non-differentiable will therefore
 determine $t(M_1)$. Subtracting the term in $g_2$ coming from $t(M_1)$, we repeat the process and find $t(M_2)$, $t(M_3)$, etc.  

For $\G$ a \emph{uniform} lattice, the lattice length spectrum is related to
the spectrum of the automorphic Laplacian in the following way. Let 
$$0=\lambda_0<\lambda_1\leq \lambda_2\leq \ldots$$ 
denote the eigenvalues of $-\Delta$ and let $\{\varphi_j\}$
be a (fixed) corresponding complete orthonormal set of eigenfunctions. For $\lambda\geq 0$, we set 
$$B(\lambda):=\sum_{j:\lambda_j=\lambda}\abs{\varphi_j(e_{n+1})}^2.$$
We call the set 
$$\{(\lambda_j,B(\lambda_j)) : B(\lambda_j)\neq 0\}$$ 
the \emph{pre-spectrum of $\Gamma$ at $e_{n+1}$}. Then, using
Selberg's pre-trace formula (see, e.g., \cite[Ch.\ XI
Sect.\ 2]{Chavel:1984a}) in a way similar to the one in the proof of  Huber's theorem  (see \cite[Thm.\ 9.2.9]{Buser:1992a}), one shows that the pre-spectrum determines and is
determined by the lattice length spectrum: If two uniform hyperbolic
lattices have the same pre-spectrum, then the right-hand sides of their pre-trace formulas will be the same for every choice of test function, and if
they have the same lattice length spectra, then the
left-hand sides of their pre-trace formulas will be the same for every choice of test function.

We summarize the above discussion as follows:

\begin{theorem} Two uniform hyperbolic lattices in $G$ have the same pair
  correlation densities and covolumes if and only if they have the
  same lattice length spectra if and only if they have the same pre-spectra at $e_{n+1}$. 
\end{theorem}

\begin{appendices}
\section{Appendix}\label{appendix}

In this appendix, we derive several elementary but essential properties of the
function $f_\xi(l)$ defined in \eqref{DEFOFF}.

\subsection{Bounds on $f_\xi$}
We recall that $A=A(l)=\cosh l$, $B=B(l)=\sinh l$ and $C=C(l)=2\sinh(l/2)$.

\begin{lemma}\label{stuff}
Let $\epsilon>0$ and let $l_0>0$ be fixed. We have, uniformly in $\xi\in\R^+$,
\begin{enumerate}[(i)]
 \item\label{stuff-one}$f_\xi(l)=O(\xi^{-n}(1+l))$,
 \item \label{stuff-onepointfive} $f_\xi(l)=O(\xi^{n-2}B^{-2(n-1)})$ uniformly in $l\geq l_0$ if $\xi\leq C$,
 \item \label{morestuff}$f_\xi(l)=O_{\epsilon}(B^{-n/2+\epsilon})$ uniformly in $l\geq l_0$.
 \item \label{stuff-two} For fixed $\xi>0$, the function $f_\xi(l)$ extends continuously to $l=0$ with value $0$.  
\end{enumerate}
\end{lemma}

\begin{proof}
To prove \eqref{stuff-one}, we notice the following general bound:
\begin{align*}
  f_\xi(l)&\leq
  \frac{2^{n-2}}{\xi^n}\int_{-1}^1\frac{(1-A/B+(y+A/B))^{n-2}}{(y+A/B)^{n-1}}\,dy\\
&=\frac{2^{n-2}}{\xi^n}\sum_{m=0}^{n-2}\binom{n-2}{m}(1-A/B)^m\int_{-1}^1(y+A/B)^{-m-1}\,dy.
\end{align*}
Computing the integral, we see that 
\begin{equation*}
(1-A/B)^m\int_{-1}^1(y+A/B)^{-m-1}\,dy=
\begin{cases}
O(1)&\textrm{if $m>0$,}\\
O(l)& \textrm{if }m=0,
\end{cases}
\end{equation*} from which the claim follows.

We next bound $f_\xi(l)$ when $\xi\leq C$. Since in particular $\xi<B$, we
may use that
\begin{equation}\label{backtoschool}
\alpha=1+O(\xi^2/B^2)
\end{equation} 
and that, for $0\leq \alpha\leq y\leq 1$, we have $(1-y^2)=(1+y)(1-y)\leq 2(1-\alpha)$. This implies 
\begin{align*}
 \xi^{-n}\int_\alpha^1\frac{(1-y^2)^{n-2}}{(y+A/B)^{n-1}}\,dy&=O\left(\xi^{-n}(1-\alpha)^{n-2}\int_\alpha^1dy\right)\\
 &=O\big(\xi^{n-2}/B^{2(n-1)}\big).
\end{align*}
It remains to analyze 
$\xi^{-n}\int_{-1}^{\lambda_-}\frac{(1-y^2)^{n-2}}{(y+A/B)^{n-1}}dy$. A small computation using
\begin{equation*}
\alpha=1-\frac{1}{2}\frac{\xi^2}{B^2}+O\Big(\frac{\xi^4}{B^4}\Big),\quad \frac{A}{B}=1+\frac{1}{2B^2}+O\Big(\frac{1}{B^4}\Big)
\end{equation*}
shows that 
\begin{equation}\label{enplus}
  1+\lambda_-=O\Big(\frac{\xi^2}{B^4}\Big). 
\end{equation}
It follows that
\begin{align*}
\xi^{-n}&\int_{-1}^{\lambda_-}\frac{(1-y^2)^{n-2}}{(y+A/B)^{n-1}}\,dy\\
&\leq\frac{2^{n-2}}{\xi^n(A/B-1)^{n-1}}\int_{-1}^{\lambda_-}(1+y)^{n-2}\,dy\\
&=O\left(\frac{(1+\lambda_-)^{n-1}}{\xi^n(A/B-1)^{n-1}}\right)=O\left(\frac{\xi^{n-2}}{B^{2(n-1)}}\right),
\end{align*}
where in the last line we have used $l\geq l_0$ to conclude that
$(A/B-1)^{-1}=O(B^{2})$. This proves \eqref{stuff-onepointfive}.

To prove \eqref{morestuff}, we use \eqref{stuff-one} and $C^2=2(A-1)$ to conclude that
when $C\leq \xi$  we have
$f_\xi(l)=O_{\epsilon}(B^{-n/2+\epsilon})$. Furthermore, we use
\eqref{stuff-onepointfive} to see that when $\xi\leq C$ we have
$f_\xi(l)=O(B^{-3n/2+1})$, which is even better.

Finally, to prove \eqref{stuff-two}, we note that
\begin{equation*}
0\leq  f_\xi(l)\leq\frac{B^{n-1}}{\xi^{n}}\int_{-1}^{1}\frac{(1-y^2)^{n-2}}{(By+A)^{n-1}}\,dy\leq
  \frac{2B^{n-1}}{\xi^{n}(A-B)^{n-1}}= \frac{2^{2-n}(e^{2l}-1)^{n-1}}{\xi^n},
\end{equation*}
from which the claim follows.
\end{proof}

\begin{remark}\label{remarkaboutstuff}
We note that Lemma \ref{stuff}~\eqref{morestuff} and \eqref{stuff-two}, together with the discreteness of $\G$, imply that
$f_\xi(t(M))=O_{\epsilon}(\norm{M}^{-n+\epsilon})$ for $M\in \G$. 
\end{remark}

\subsection{Asymptotics}

In this section, we analyze $\int_Gf_\xi(t(g))\,dg$.

\begin{theorem}\label{theoremonasymptotics}
Let $\epsilon>0$. We have, for $\xi\geq 1$,
\begin{enumerate}[(i)]
\item \label{asymptotics}\quad   $\displaystyle \int_Gf_\xi(t(g))\,dg=\frac{\omega_n}{(n-1)^2}\xi^{n-2}+O_{\epsilon}\big(\xi^{-1+\epsilon}+\xi^{n-4}\big),$
\item  \label{upperbounds} \quad  $\displaystyle\int_Gf_\xi(t(g))\norm{g}^{-\alpha}\,dg=O_{\epsilon}\big(\xi^{n-2-\alpha+\epsilon}\big),$
\end{enumerate}
for $0\leq \alpha\leq 2(n-1) $.
\end{theorem}

\begin{proof}
We let $l_1=l_1(\xi)\leq  l_2=l_2(\xi)$ be defined by $B(l_1)=C(l_2)=\xi$. By Proposition \ref{KAKmeasure} we have, for $\alpha\geq 0$, that
\begin{equation*}
\int_Gf_\xi(t(g))\norm{g}^{-\alpha}\,dg=\omega_n2^{-\alpha/2}\int_0^\infty f_\xi(l)A^{-\alpha/2}B^{n-1}\,dl.
\end{equation*}
Using Lemma \ref{stuff}\eqref{stuff-one}, we see that 
\begin{equation*}
  \int_0^{l_2}f_\xi(l)A^{-\alpha/2}B^{n-1}\,dl=O_{\epsilon}\big(\xi^{n-2-\alpha+\epsilon}\big).
\end{equation*}
Also from Lemma \ref{stuff}\eqref{stuff-one}, we find $\int_{0}^{l_1}f_\xi(l)B^{n-1}\,dl=O_{\epsilon}(\xi^{-1+\epsilon}).$
Furthermore, Lemma \ref{stuff}\eqref{stuff-onepointfive} gives
\begin{equation*}
  \int_{l_2}^\infty f_\xi(l)A^{-\alpha/2}B^{n-1}\,dl=O\left(\xi^{n-2}\int_{l_2}^\infty B^{-(n-1+\alpha/2)}\,dl\right)=O(\xi^{-n-\alpha}).
\end{equation*}
Collecting the above bounds, we immediately arrive at the claim in \eqref{upperbounds}.

To get the asymptotics in \eqref{asymptotics}, we must understand $f_\xi(l)$ better in the regime $l_1\leq l\leq l_2$. By the proof of Lemma \ref{stuff}\eqref{stuff-onepointfive}, we see again that 
\begin{equation*}
\frac{1}{\xi^n}\int_{[-1,\lambda_-)\cup(\alpha,1]}\frac{(1-y^2)^{n-2}}{(y+A/B)^{n-1}}\,dy=O\big(\xi^{n-2}/B^{2(n-1)}\big),
\end{equation*}
and integrating this times $B^{n-1}$ from $l_1$ to $l_2$ we get the bound $O(\xi^{-1})$. Combining this with the above bounds gives
\begin{equation*}
\int_Gf_\xi(t(g))\,dg=\omega_n\int_{l_1}^{l_2}
\xi^{-n}\int_{\lambda_+}^{-\alpha}\frac{(1-y^2)^{n-2}}{(y+\frac{A}{B})^{n-1}}\,dy
B^{n-1}dl+O_{\epsilon}\big(\xi^{-1+\epsilon}\big).
\end{equation*}

We now analyze 
$\xi^{-n}\int_{\lambda_+}^{-\alpha}\frac{(1-y^2)^{n-2}}{(y+\frac{A}{B})^{n-1}}\,dy$. 
Using the identity
\begin{equation*}
  1-y^2=\big((1+\tfrac{A}{B})-(y+\tfrac{A}{B})\big)\big((1-\tfrac{A}{B})+(y+\tfrac{A}{B})\big)
\end{equation*}
and writing $n'=n-2$, we find that
\begin{align}\label{2timesbinomial}
  &(1-y^2)^{n'}=\sum_{j_1,j_2=0}^{n'}\binom{n'}{j_1}\binom{n'}{j_2}\big(1+\tfrac{A}{B}\big)^{n'-j_1}\big(1-\tfrac{A}{B}\big)^{j_2}(-1)^{j_1}\big(y+\tfrac{A}{B}\big)^{n'+j_1-j_2}\\
&=\sum_{j_1,j_2=0}^{n'}\!\!\sum_{k=0}^{n'-j_1}\!\!\binom{n'}{j_1}\binom{n'}{j_2}\binom{n'-j_1}{k}2^{n'-j_1-k}\big(1-\tfrac{A}{B}\big)^{j_2+k}(-1)^{j_1+k}\big(y+\tfrac{A}{B}\big)^{n'+j_1-j_2},\nonumber
\end{align}
so we investigate integral expressions of the form
\begin{equation*}
  \frac{1}{\xi^n}\big(1-\tfrac{A}{B}\big)^{j_2+k}\int_{\lambda_+}^{-\alpha}\big(y+\tfrac{A}{B}\big)^{j_1-j_2-1}\,dy.
\end{equation*}
A trivial upper bound for these expressions is
\begin{align*}
\frac{B^{-2(j_2+k)}}{\xi^n}&\int_{-1}^0\big(y+\tfrac{A}{B}\big)^{j_1-j_2-1}\,dy\\ 
& \ll\frac{B^{-2(j_2+k)}}{\xi^n}\begin{cases}1 &\textrm{ if }j_1-j_2-1\geq 0,\\
|\log(A/B-1)| &\textrm{ if }j_1-j_2-1=-1,\\
\big(\frac{A}{B}-1\big)^{j_1-j_2} &\textrm{ if }j_1-j_2-1<-1,
\end{cases}
\end{align*}
which gives
\begin{align*}
\int_{l_1}^{l_2}&\frac{1}{\xi^n}\big(1-\tfrac{A}{B}\big)^{j_2+k}\int_{\lambda_+}^{-\alpha}\big(y+\tfrac{A}{B}\big)^{j_1-j_2-1}\,dyB^{n-1}dl\\
&\ll_{\epsilon}\xi^{-n}\begin{cases}\int_{l_1}^{l_2}B^{-2(j_2+k)+n-1}\,dl &\textrm{ if }j_1-j_2-1\geq 0,\\
\int_{l_1}^{l_2}B^{-2(j_2+k)+n-1+\epsilon}\,dl &\textrm{ if }j_1-j_2-1=-1,\\
\int_{l_1}^{l_2}B^{-2(j_2+k)-2(j_1-j_2)+n-1}\,dl &\textrm{ if }j_1-j_2-1<-1,
\end{cases}
\\ 
&\ll\begin{cases}
\xi^{-2(j_2+k)-1+\epsilon}+\xi^{-4(j_2+k)+n-2+\epsilon}&\textrm{ if }j_1-j_2-1\geq-1, \\
\xi^{-2(j_1+k)-1}+\xi^{-4(j_1+k)+n-2}&\textrm{ if }j_1-j_2-1<-1,
\end{cases}
\end{align*}
where we have used that $B(l_1)=\xi$ and $B(l_2)^2=\frac{\xi^4}{4}+\xi^2$ (which follows from $C(l_2)=\xi$).
We notice that the above expression is $O(\xi^{n-4})$ unless
$j_1-j_2-1\geq -1$ and $j_2+k=0$ (which is the case $k=j_2=0$ and $j_1=0,\ldots,
n-2$ ), or $j_1-j_2-1< -1$ and $j_1+k=0$ (which is the case $k=j_1=0$ and $j_2=1,\ldots,
n-2$).

We now find a slightly less trivial bound when $k=j_2=0$ and $j_1=1,\ldots,
n-2$. We have, since $-\alpha+\frac{A}{B}$ and
$\lambda_++\frac{A}{B}$ are bounded from above, that
\begin{align*}
\frac{1}{\xi^n}&\int_{\lambda_+}^{-\alpha}\big(y+\tfrac{A}{B}\big)^{j_1-1}\,dy=
\frac{1}{j_1\xi^n}\left(\big(-\alpha+\tfrac{A}{B}\big)^{j_1}-\big(\lambda_++\tfrac{A}{B}\big)^{j_1}\right)\\
&=\frac{1}{j_1\xi^n}\big(-\alpha+\tfrac{A}{B}-\big(\lambda_++\tfrac{A}{B}\big)\big)\sum_{v=0}^{j_1-1}\left(-\alpha+\tfrac{A}{B}\right)^v\left(\lambda_++\tfrac{A}{B}\right)^{j_1-1-v}\\
&=O\left(\frac{\abs{\lambda_++\alpha}}{\xi^n}\right).
\end{align*}
We observe that 
\begin{align*}
\lambda_++\alpha&=\frac{-\xi^2\frac{A}{B}+\sqrt{1-\frac{\xi^2}{B^2}}+(\xi^2+1)\sqrt{1-\frac{\xi^2}{B^2}}}{\xi^2+1}\\
&=\frac{\xi^2}{1+\xi^2}\bigg(\sqrt{1-\frac{\xi^2}{B^2}}-\frac{A}{B}\bigg)+\frac{2}{\xi^2+1}\sqrt{1-\frac{\xi^2}{B^2}}\\
&=\frac{\xi^2}{1+\xi^2}\big(1+O(\xi^2/B^2)-1+O(B^{-2})\big)+O(\xi^{-2})\\
&=O\big(\xi^2/B^2+B^{-2}+\xi^{-2}\big)
\end{align*}
 and hence
\begin{align*}
\int_{l_1}^{l_2}&\frac{1}{\xi^n}\int_{\lambda_+}^{-\alpha}\big(y+\tfrac{A}{B}\big)^{j_1-1}\,dy B^{n-1}dl \\
&=O\left(\xi^{-n}\int_{l_1}^{l_2}\big(\xi^{2} B^{n-3}+ B^{n-3}+\xi^{-2}B^{n-1}\big)\,dl\right)=O_\epsilon\big(\xi^{n-4}+\xi^{-1+\epsilon}\big).
\end{align*}

We next consider the contribution when $k=j_1=0$ and $j_2=1,\ldots,
n-2$. We find that
\begin{align*}
\frac{1}{\xi^n}&\big(1-\tfrac{A}{B}\big)^{j_2}\int_{\lambda_+}^{-\alpha}\big(y+\tfrac{A}{B}\big)^{-j_2-1}\,dy\\
&=\frac{1}{-j_2\xi^n}\big(1-\tfrac{A}{B}\big)^{j_2}\left(\big(-\alpha+\tfrac{A}{B}\big)^{-j_2}-\big(\lambda_++\tfrac{A}{B}\big)^{-j_2}\right)\\
&=O\left(\frac{B^{-2j_2}}{\xi^n}\left(\frac{B^{2j_2}}{\xi^{2j_2}}+\xi^{2j_2}\right)\right)=O\left(\xi^{-n-2j_2}+\xi^{2j_2-n}B^{-2j_2}\right),
\end{align*}
where we have used 
\begin{equation*}
\lambda_++A/B=\frac{\sqrt{1-\xi^2/B^2}+A/B}{\xi^2+1}\geq
\frac{1}{\xi^2+1}\geq \tfrac12\xi^{-2}, 
\end{equation*}
 and 
\begin{align}\nonumber
  A/B-\alpha&=\frac{A-B\sqrt{1-\xi^2/B^2}}{B}=\frac{A-B+B(1-\sqrt{1-\xi^2/B^2})}{B}\\
&\geq \frac{A-B}{B}+\frac{\xi^2}{2B^2}\geq \frac{\xi^2}{2B^2}\label{yetanotherbound}
\end{align}
(which holds since $1-\sqrt{1-x^2}\geq x^2/2$). Hence
\begin{align*}
&\int_{l_1}^{l_2}\frac{1}{\xi^n}\big(1-\tfrac{A}{B}\big)^{j_2}\int_{\lambda_+}^{-\alpha}\big(y+\tfrac{A}{B}\big)^{-j_2-1}\,dy B^{n-1}dl\\
&=O\left(\int_{l_1}^{l_2}\left(\xi^{-n-2j_2}B^{n-1}+\xi^{2j_2-n}B^{n-2j_2-1}\right)\,dl\right)=O_\epsilon\left(\xi^{n-2-2j_2}+\xi^{-1+\epsilon}\right). 
\end{align*}

The only term we have not analyzed so far is the one with $j_1=j_2=k=0$
corresponding to the function
\begin{equation*}
h_\xi(l)=  \frac{2^{n-2}}{\xi^n}\int_{\lambda_+}^{-\alpha}\frac{1}{y+\frac{A}{B}}\,dy=\frac{2^{n-2}}{\xi^{n}}\log\left(\frac{-\alpha+\frac AB}{\lambda_++\frac AB}\right).
\end{equation*}
We notice that 
\begin{equation*}
\frac{-\alpha+\frac
  A B}{\lambda_++\frac
  A B}=(\xi^2+1)\frac{-\alpha+\frac
  A B}{\alpha+\frac
  A B}=\frac{(\xi^2+1)^2}{(A+B\alpha)^2}
\end{equation*} 
and it follows that, for $l_1\leq l\leq l_2$, we have
$h_\xi(l)=O(\frac{l+\log(\xi)}{\xi^n})$. Using this fact, together with $A-B=O(B^{-1})$, we see that 
\begin{align*}
\int_{l_1}^{l_2}&h_\xi(l)B^{n-1}\,dl=\int_{l_1}^{l_2}h_\xi(l)B^{n-2}A\,dl+O_{\epsilon}\big(\xi^{n-5+\epsilon}\big)\\
&=\left.h_\xi(l)\frac{B^{n-1}}{n-1}\right |_{l_1}^{l_2}-\frac{1}{n-1}\int_{l_1}^{l_2}h'_\xi(l)B^{n-1}\,dl+O_{\epsilon}\big(\xi^{n-5+\epsilon}\big)\\
&=h_\xi(l_2)\frac{B^{n-1}(l_2)}{n-1}+\frac{2^{n-1}}{(n-1)\xi^n}\int_{l_1}^{l_2}\frac{B^{n}}{\sqrt{B^2-\xi^2}}\,dl+O_{\epsilon}\big(\xi^{n-5+\epsilon}+\xi^{-1+\epsilon}\big).
\end{align*}
Using $B(l_2)^2=\frac{\xi^4}{4}+\xi^2$, 
we see that $A+B\alpha |_{l_2}=1+\xi^2$,
from which it follows that $h_\xi(l_2)=0$. We conclude that 
\begin{align*}
\int_{l_1}^{l_2}h_\xi(l)B^{n-1}\,dl&=\frac{2^{n-1}}{(n-1)\xi^n}\int_{l_1}^{l_2}\frac{B^{n}}{\sqrt{B^2-\xi^2}}\,dl+O_{\epsilon}\big(\xi^{n-5+\epsilon}+\xi^{-1+\epsilon}\big).
\end{align*}

To compute the last integral, we first notice that for any $m\geq 2$, we have
\begin{align*}
\int_{l_1}^{l_2}\frac{B^{m}}{\sqrt{B^2-\xi^2}}\,dl&\ll
B^{m-2}(l_2)\int_{l_1}^{l_2}\frac{AB}{\sqrt{B^2-\xi^2}}\,dl\\
&\ll \left.\xi^{2(m-2)}\sqrt{B^2-\xi^2}\right|_{l_1}^{l_2}\ll \xi^{2(m-1)}.
\end{align*}
It follows that
\begin{align*}
&\frac{1}{\xi^n}\int_{l_1}^{l_2}\frac{B^{n}}{\sqrt{B^2-\xi^2}}\,dl=\frac{1}{\xi^n}\int_{l_1}^{l_2}\frac{B^{n-1}A}{\sqrt{B^2-\xi^2}}\,dl+O\big(\xi^{n-4}\big)\\
&=\left.\frac{B^{n-2}}{\xi^n}\sqrt{B^2-\xi^2}\right|_{l_1}^{l_2}-\frac{(n-2)}{\xi^n}\int_{l_1}^{l_2}B^{n-3}A\sqrt{B^2-\xi^2}\,dl+O\big(\xi^{n-4}\big)\\
&=\left.\frac{B^{n-2}}{\xi^n}\sqrt{B^2-\xi^2}\right|_{l_1}^{l_2}-\frac{(n-2)}{\xi^n}\int_{l_1}^{l_2}B^{n-2}A\,dl+O\left(\xi^{2-n}\int_{l_1}^{l_2}B^{n-4}A\,dl\right)+O\big(\xi^{n-4}\big)\\
&=\frac{\xi^{n-2}}{2^{n-1}}-\frac{n-2}{n-1}\frac{\xi^{n-2}}{2^{n-1}}+O_{\epsilon}\big(\xi^{-1+\epsilon}+\xi^{n-4}\big)=\frac{1}{2^{n-1}(n-1)}\xi^{n-2}+O_{\epsilon}\big(\xi^{-1+\epsilon}+\xi^{n-4}\big).
\end{align*}
 This finally allows us to conclude that 
\begin{equation*}
  \int_Gf_\xi(t(g))\,dg=\frac{\omega_n}{(n-1)^2}\xi^{n-2}+O_{\epsilon}\big(\xi^{-1+\epsilon}+\xi^{n-4}\big),
\end{equation*}
which completes the proof.
\end{proof}

\subsection{Bounds on the derivative of $f_\xi$}

\begin{lemma}\label{lemmafderivative}

For any  $0< \delta \leq 1$, any $\xi\geq 1$ and any $l> 0$, we have
\begin{equation*}
f'_\xi(l)\ll\begin{cases}
\xi^{-n}&\textrm{ if }l<  l_1(\xi),\\
\xi^{-n}\delta^{-1/2}&\textrm{ if }  l_1(\xi)+\delta\leq l<l_1(\xi)+1,\\
\xi^{-n}&\textrm{ if  } l_1(\xi)+1\leq l< l_2(\xi),\\
\xi^{n-2}B^{-2(n-1)}&\textrm{ if  } l_2(\xi)<l.\\
\end{cases}
\end{equation*}
\end{lemma}

\begin{proof}
Let, for $-1\leq x\leq1$,
$$F(x,l)=\int_{-1}^x\frac{\left(1-y^2\right)^{n-2}}{\left(y+\frac AB\right)^{n-1}}\,dy.$$ 
The partial derivatives satisfies
\begin{align*}
\frac{\partial F}{\partial x}(x,l)&=F_1(x,l)=\frac{\left(1-x^2\right)^{n-2}}{\left(x+\frac AB\right)^{n-1}},\\
\frac{\partial F}{\partial l}(x,l)&=F_2(x,l)=\frac{n-1}{B^2}\int_{-1}^x\frac{\left(1-y^2\right)^{n-2}}{\left(y+\frac AB\right)^{n}}\,dy.
\end{align*}
We notice that $F_2$ is bounded. Indeed, we have
\begin{align*}
F_2(x,l)&\leq \frac{2^{n-2}(n-1)}{B^2}\int_{-1}^1\frac{\left(1+y\right)^{n-2}}{\left(y+\frac AB\right)^{n}}\,dy\nonumber\\
&=\frac{2^{n-2}(n-1)}{B^2}\sum_{k=0}^{n-2}\binom{n-2}{k}\big(1-\tfrac{A}{B}\big)^k\int_{-1}^1{\left(y+\tfrac AB\right)^{-(k+2)}}\,dy\nonumber\\
&=\frac{2^{n-2}(n-1)}{B}\sum_{k=0}^{n-2}\binom{n-2}{k}\frac{(B-A)^k}{k+1}\left((A+B)^{k+1}-(A-B)^{k+1}\right),\nonumber
\end{align*}
which is $O(1)$ since $\frac{(B-A)^k}{B}\left((A+B)^{k+1}-(A-B)^{k+1}\right)$ is uniformly bounded. Since we have $f'_\xi(l)=\xi^{-n}F_2(1,l)$ for $l< l_1(\xi)$, we immediately get the claimed bound in this region.

For $l> l_2(\xi)$, we have 
\begin{equation*}
  f'_\xi(l)=\xi^{-n}\left(F_1(\lambda_-,l)\lambda_-'+F_2(\lambda_-,l)-F_1(\alpha,l)\alpha'+F_2(x,l)\vert_{x=\alpha}^{x=1}\right).
\end{equation*}
A small computation shows that
\begin{equation}\label{derivatives}
\alpha'=\frac{\xi^2}{B^2}\frac{A/B}{\alpha},\quad
\lambda_{\pm}'=\frac{\frac{\xi^2}{B^2}\big(1\pm \frac{A/B}{\alpha}\big)}{1+\xi^2}.
\end{equation}
Since $\xi$ is bounded away from zero and $l>l_2(\xi)$, we see that
also $l$ is bounded away from zero; in particular $A/B$ is bounded. In addition,
we have $\xi^2<C^2=2(A-1)$. Combining these facts, we see that $ \alpha=\sqrt{1-\xi^2/B^2}$ is bounded away from
zero and that $\xi^2=O(B)$. Using also \eqref{backtoschool}
and the relation $AB^{-1}-1\asymp B^{-2}$, it follows that when $l>l_2(\xi)$ 
\begin{equation*}
\lambda_{-}'=O\left(B^{-2}\big(1-A/B+A/B(1-1/\alpha)\big)\right)=O\left(\frac{\xi^2}{B^{4}}\right).
\end{equation*}
Next, using \eqref{enplus}, we see that 
\begin{equation*}
F_1(\lambda_-,l)=O\left(\frac{(1+\lambda_-)^{n-2}}{(\frac{A}{B}-1)^{n-1}}\right)=O\left(\frac{\xi^{2(n-2)}}{B^{2(n-3)}}\right),
\end{equation*}
and it follows that 
\begin{equation}\label{bound1}
\xi^{-n}F_1(\lambda_-,l)\lambda_-' 
=O\left(\frac{\xi^{n-2}}{B^{2(n-1)}}\right).
\end{equation}

To bound $\xi^{-n}F_2(\lambda_-,l)$, we observe that
\begin{align*}
F_2(\lambda_-,l)&\leq
\frac{2^{n-2}(n-1)}{B^2(\frac{A}{B}-1)^n}\int_{-1}^{\lambda_-}(1+y)^{n-2}\,dy\\
&=O\big(B^{2n-2}(1+\lambda_-)^{n-1}\big)
=O\left(\frac{\xi^{2(n-1)}}{B^{2(n-1)}}\right),
\end{align*}
where again we have used \eqref{enplus}. It follows that 
\begin{equation}\label{bound2}
\xi^{-n}F_2(\lambda_-,l)=O\left(\frac{\xi^{n-2}}{B^{2(n-1)}}\right).
\end{equation}

We can bound $\xi^{-n}F_1(\alpha,l)\alpha'$ by using \eqref{backtoschool} to get 
\begin{equation*}
F_1(\alpha,l)\leq (1-\alpha^2)^{n-2}
=O\left(\frac{\xi^{2(n-2)}}{B^{2(n-2)}}\right).
\end{equation*}
Combining this with $\alpha'=O(\frac{\xi^2}{B^2})$, we find that
\begin{equation}\label{bound3}  
\xi^{-n}F_1(\alpha,l)\alpha'=O\left(\frac{\xi^{n-2}}{B^{2(n-1)}}\right).
\end{equation}
Finally, again using \eqref{backtoschool}, we see that               
\begin{align*}
F_2&(x,l)\vert_{x=\alpha}^{x=1}=\frac{n-1}{B^2}\int_\alpha^1\frac{(1-y^2)^{n-2}}{(y+\frac{A}{B})^n}\,dy\\
&=O\left(\frac{1}{B^2}\int_\alpha^1(1-y)^{n-2}\,dy\right)=O\left(\frac{(1-\alpha)^{n-1}}{B^2}\right)=O\left(\frac{\xi^{2(n-1)}}{B^{2n}}\right),
\end{align*}
and it follows that 
\begin{equation}\label{bound4}
\xi^{-n}F_2(x,l)\vert_{x=\alpha}^{x=1}=O\left(\frac{\xi^{n-2}}{B^{2n}}\right).
\end{equation}
Combining \eqref{bound1}, \eqref{bound2}, \eqref{bound3}, and
\eqref{bound4}, we get the desired bound when $l> l_2(\xi)$.

For $l_1(\xi)<l< l_2(\xi)$, we have
\begin{align*}
f'_\xi(l)=\xi^{-n}\Big(F_1(\lambda_-,l)\lambda_-'&+F_2(\lambda_-,l)-F_1(-\alpha,l)\alpha'-F_1(\lambda_+,l)\lambda_+'\\
&+F_2(x,l)\vert_{x=\lambda_+}^{x=-\alpha}-F_1(\alpha,l)\alpha'+F_2(x,l)\vert_{x=\alpha}^{x=1}\Big).
\end{align*}
Since $F_2(x,l)$ is bounded, all terms involving $F_2(x,l)$ are
$O(\xi^{-n})$. We also observe that 
\begin{equation*}
F_1(x,l)\leq \frac{2^{n-2}}{x+\frac{A}{B}}.
\end{equation*}
It follows, using \eqref{yetanotherbound}, that $F_1(\pm\alpha,l)=O((A/B-\alpha)^{-1})=O(\frac{B^2}{\xi^2})$. In the same way we see, using
\begin{align*}
\lambda_-+\tfrac{A}{B}=\frac{\frac{A}{B}-\sqrt{1-{\frac{\xi^2}{B^2}}}}{\xi^2+1}&\geq\frac{1-\sqrt{1-\frac{\xi^2}{B^2}}}{\xi^2+1}\geq \frac{\xi^2}{2B^2(\xi^2+1)},
\end{align*}
that $F_1(\lambda_\pm,l)=O((\lambda_-+\frac{A}{B})^{-1})=O(B^2)$. This finishes the proof once we observe, using \eqref{derivatives}, that \begin{equation*}
\alpha' \ll\begin{cases}
\frac{\xi^2}{\sqrt \delta B^2}&\textrm{ if } l_1(\xi)+\delta\leq l<l_1(\xi)+1,\\
\frac{\xi^2}{ B^2}& \textrm{ if }l_1(\xi)+1\leq l<l_2(\xi),
\end{cases}
\end{equation*}
\begin{equation*}
\lambda_{\pm}' \ll\begin{cases}
\frac{1}{\sqrt \delta B^2}&\textrm{ if } l_1(\xi)+\delta \leq l<l_1(\xi)+1,\\
\frac{1}{ B^2}& \textrm{ if }l_1(\xi)+1\leq l<l_2(\xi).
\end{cases}
\end{equation*}
Here we have used the fact that for $l_1(\xi)+\delta\leq l<l_2(\xi)$, we have $B\geq
\xi e^{\delta}$, from which it follows that $\alpha^{-1}=O(\delta^{-{1/2}})$. In particular, when $l\geq l_1(\xi)+1$, we have $B\geq\xi e$, from which it follows that $\alpha^{-1}$ is bounded.
\end{proof}

Next, we use Lemma \ref{lemmafderivative} to determine how $f_{\xi}(l)$ behaves under small changes in the variable $l$. 

\begin{lemma}\label{perturbationlemma}
For any sufficiently small $\delta>0$, any sufficiently large $\xi$, and any $l,l'>0$ satisfying $|l-l'|<\delta$, we have
\begin{equation*}
|f_{\xi}(l)-f_{\xi}(l')|\ll
\begin{cases}
\delta\xi^{-n}& \text{if $l\leq l_1(\xi)-\delta$},\\
\delta^{1/2}\xi^{-n}& \text{if $l_1(\xi)-\delta<l<l_1(\xi)+1$},\\
\delta\xi^{-n}& \text{if $l_1(\xi)+1\leq l<l_2(\xi)+\delta$},\\
\delta\xi^{n-2}B^{-2(n-1)}& \text{if $l_2(\xi)+\delta\leq l$}.
\end{cases}
\end{equation*}
\end{lemma}

\begin{proof}
To prove the claim, we first observe that all estimates except for (part of) the second one follows directly from Lemma \ref{lemmafderivative} and the mean value theorem. It remains to study the case $l_1(\xi)-\delta<l<l_1(\xi)+2\delta$. Here we use the rough bound
\begin{equation}\label{theroughbound}
|f_{\xi}(l)-f_{\xi}(l')|\leq|f_{\xi}(l_1(\xi))-f_{\xi}(l_1(\xi)-2\delta)|+|f_{\xi}(l_1(\xi)+3\delta)-f_{\xi}(l_1(\xi))|,
\end{equation} 
 valid for $\delta$  sufficiently small. This follows from $f_\xi$
 being increasing in $l\leq l_1(\xi)$ and decreasing for
 $l\in[l_1(\xi),l_1(\xi)+3\delta]$ with $\delta$ sufficiently small.  It follows from Lemma \ref{lemmafderivative} that the first term on the right-hand side of \eqref{theroughbound} is bounded by $O(\delta\xi^{-n})$. 

For the rest of the proof, the symbols $A$, $B$, $\alpha$ and $\lambda_{\pm}$ will always be evaluated at $(\xi,l)=(\xi,l_1(\xi)+3\delta)$. In addition, we set $A'=A(l_1(\xi))$ and $B'=B(l_1(\xi))$. It follows readily from \eqref{DEFOFF} that the second term on the right-hand side of \eqref{theroughbound} is bounded by 
\begin{align}
\nonumber&\frac1{\xi^n}\int_{-1}^{1}(1-y^2)^{n-2}\left(\big(y+\tfrac AB\big)^{-(n-1)}-\big(y+\tfrac{A'}{B'}\big)^{-(n-1)}\right)\,dy\\
&+\frac1{\xi^n}\int_{\lambda_-}^{\lambda_+}\frac{(1-y^2)^{n-2}}{(y+\frac{A'}{B'})^{n-1}}\,dy
+\frac1{\xi^n}\int_{-\alpha}^{\alpha}\frac{(1-y^2)^{n-2}}{(y+\frac{A'}{B'})^{n-1}}\,dy.\nonumber
\end{align}
We call the terms above $I_1$, $I_2$ and $I_3$ respectively and recall that we need to show that $I_j=O(\delta^{1/2}\xi^{-n})$ for $1\leq j\leq3$.

We now estimate the integral $I_1$ using the same basic idea as in the proof of Theorem \ref{theoremonasymptotics}. We notice that 
\begin{align*}
\frac{A'}{B'}-\frac AB=O\left(\frac{\delta}{\xi^2}\right),
\end{align*} 
and hence we also obtain 
\begin{align*}
\big(y+\tfrac AB\big)^{-(n-1)}-\big(y+\tfrac{A'}{B'}\big)^{-(n-1)}=O\bigg(\frac{\frac{A'}{B'}-\frac AB}{(y+\frac{A}{B})^{n}}\bigg)=O\bigg(\frac{\delta}{\xi^2(y+\frac{A}{B})^{n}}\bigg).
\end{align*}
It follows that 
\begin{align*}
I_1=O\left(\frac{\delta}{\xi^{n+2}}\right)\int_{-1}^1\frac{(1-y^2)^{n-2}}{(y+\frac{A}{B})^{n}}\,dy.
\end{align*}
Hence, expanding the numerator as in \eqref{2timesbinomial}, we investigate expressions of the form
\begin{equation*}
\frac{\delta}{\xi^{n+2}}\big(1-\tfrac{A}{B}\big)^{j_2+k}\int_{-1}^{1}\big(y+\tfrac{A}{B}\big)^{j_1-j_2-2}\,dy.
\end{equation*}
Using straightforward estimates, we obtain the bound
\begin{align*}
O\left(\frac{\delta}{\xi^{n+2}B^{2(j_2+k)}}\right)
&\begin{cases}
1& \textrm{if $j_1-j_2-2\geq0$,}\\
|\log(A/B-1)|& \textrm{if $j_1-j_2-2=-1$,}\\
\big(\frac{A}{B}-1\big)^{j_1-j_2-1}& \textrm{if $j_1-j_2-2<-1$,}
\end{cases}\\
=O_{\epsilon}\left(\frac{\delta}{\xi^{n+2}}\right)
&\begin{cases}
B^{-2(j_2+k)}& \textrm{if $j_1-j_2-2\geq0$,}\\
B^{-2(j_2+k)+\epsilon}& \textrm{if $j_1-j_2-2=-1$,}\\
B^{-2(j_1+k-1)}& \textrm{if $j_1-j_2-2<-1$,}
\end{cases}\\
=O_{\epsilon}\left(\delta\right)
&\begin{cases}
\xi^{-n-2(j_2+k+1)}& \textrm{if $j_1-j_2-2\geq0$,}\\
\xi^{-n-2(j_2+k+1)+\epsilon}& \textrm{if $j_1-j_2-2=-1$,}\\
\xi^{-n-2(j_1+k)}& \textrm{if $j_1-j_2-2<-1$.}
\end{cases}
\end{align*}
Since this is clearly $O(\delta\xi^{-n})$ in all cases, it only remains to estimate $I_2$ and $I_3$. We observe that $\alpha=O(\delta^{1/2})$. Hence $\lambda_+-\lambda_-=O(\delta^{1/2}\xi^{-2})$ and $1+\lambda_+=O(\xi^{-2})$. It follows that
\begin{align*}
I_2=O\bigg(\frac{(\lambda_+-\lambda_-)(1+\lambda_+)^{n-2}}{\xi^n(\frac{A'}{B'}-1)^{n-1}}\bigg)=O\bigg(\frac{\delta^{1/2}}{\xi^n}\bigg).
\end{align*}
Finally, using again that $\alpha=O(\delta^{1/2})$, we immediately find that also $I_3=O(\delta^{1/2}\xi^{-n})$. This finishes the proof.
\end{proof}

\begin{lemma} \label{ob} 
Fix $l>0$. Then the function $f_{\xi}(l)$ is $C^1$ in $\xi$ precisely when $\xi\notin\{2\sinh(l/2), \sinh l\}$.
\end{lemma}

\begin{proof}
 Consider the
function $h(\xi)=\xi^nf_{\xi}(l)$. Clearly the claim of the lemma is
equivalent to $h(\xi)$ being  $C^1$ precisely when
$\xi\notin\{2\sinh(l/2), \sinh l\}$. Using the same notation as in the
proof of
Lemma~\ref{lemmafderivative}, it follows from \eqref{DEFOFF} that away from
$\xi\in \{C,B\}=\{2\sinh(l/2), \sinh l\}$ the function $h$ is smooth with derivative
\begin{equation*}
  h'(\xi)=\begin{cases}-\big(F_1(\alpha,l)\alpha_\xi-F_1(\lambda_-,l)(\lambda_-)_\xi\big)& \textrm{if $\xi<C$,}\\
-\big(F_1(\lambda_+,l)(\lambda_+)_\xi-F_1(\lambda_-,l)(\lambda_-)_\xi\big)-\big(F_1(\alpha,l)+F_1(-\alpha,l)\big)\alpha_\xi& \textrm{if $C<\xi< B$,}\\
0& \textrm{if $B<\xi$.}\\
\end{cases}
\end{equation*}
Here
\begin{align*}
\alpha_\xi=\frac{d\alpha}{d\xi}&=-\frac{\xi}{B^2(1-\xi^2/B^2)^{1/2}},\\
(\lambda_{\pm})_\xi=\frac{d\lambda_{\pm}}{d\xi}&=\frac{-2\xi A/B\pm\alpha_\xi(\xi^2+1)-2\xi(\pm\alpha)}{(\xi^2+1)^2}.
\end{align*}
These quantities are defined when $\xi<B$. We notice that $\alpha_\xi,
(\lambda_{+})_\xi<0$ and that
$\alpha_\xi\to-\infty$ and $(\lambda_{\pm})_\xi\to\mp\infty$ as $\xi\to B$. 
Combining this with $F_1(x,l)>0$, we find that $h'(\xi)\to+\infty$ as
$\xi\to B$ from below, while $h'(\xi)\to 0$ as $\xi\to B$ from above. 
 This shows that $h$
is not continuously differentiable at $\xi=B=\sinh(l)$. To conclude the
same at $\xi=C=2\sinh(l/2)$, we note that 
\begin{equation*}
h'(C+\epsilon)-h'(C-\epsilon)\to
\left.\big(-F_1(\lambda_+,l)(\lambda_+)_\xi-F_1(-\alpha,l)\alpha_\xi\big)\right\vert_{\xi=C}\quad\textrm{as  }\epsilon \to 0,
\end{equation*} 
which by the above considerations is strictly
positive. The claim follows.
\end{proof}
\end{appendices}

\bibliographystyle{abbrv}
\def\cprime{$'$}

\end{document}